\documentclass[12pt]{amsart}
\usepackage{amssymb,amsmath,amsfonts,amscd}
\usepackage{times,graphicx,epsfig}
\usepackage{soul}
\usepackage{pgf,tikz}
\usepackage{mathrsfs}
\usetikzlibrary{arrows}
\usepackage{hyperref}
\hypersetup{pdfpagemode=FullScreen,colorlinks=true}

\theoremstyle{plain}
\newtheorem{theorem}{Theorem}[section]

\newtheorem{proposition}[theorem]{Proposition}
\newtheorem{definition}[theorem]{Definition}
\newtheorem{notation}[theorem]{Notation}
\newtheorem{corollary}[theorem]{Corollary}
\newtheorem{remark}[theorem]{Remark}
\newtheorem{lemma}[theorem]{Lemma}
\newtheorem{example}[theorem]{Example}

\def\eps{\epsilon}
\def\Rumin{Rumin }
\def\TFF{T_{FF}}
\def\TR{T_{R}}
\def\TAK{T_{AK}}
\def\FF{\mathfrak{FF}}
\def\Ru{\mathfrak{Ru}}
\def\BB{\mathfrak{B}}

\usepackage{pdfsync}

\usepackage{flexisym}

\newcommand{\E}{\mathbb E}

\newcommand{\rn}[1]{{\mathbb R}^{#1}}
\newcommand{\R}{\mathbb R}

\newcommand{\Z}{\mathbb Z}

\newcommand{\cov}[1]{{\bigwedge\nolimits^{#1}{\mfrak h}}}

\newcommand{\covw}[2]{{\bigwedge\nolimits^{#1,#2}{\mfrak h}}}

\newcommand{\covH}[1]{{\bigwedge\nolimits^{#1}{\mfrak h}}}

\newcommand{\covh}[1]{{\bigwedge\nolimits^{#1}{\mfrak h_1}}}

\newcommand{\he}[1]{{\mathbb H}^{#1}}

\newcommand{\scal}[2]{\langle {#1} , {#2}\rangle}

\newcommand{\Scal}[2]{\langle {#1} \vert {#2}\rangle}
\newcommand{\scalp}[3]{\langle {#1} , {#2}\rangle_{#3}}

\newcommand{\mc}{\mathcal }

\newcommand{\mfrak}{\mathfrak}

\newcommand{\res}{\mathop{\hbox{\vrule height 7pt width .5pt depth 0pt
\vrule height .5pt width 6pt depth 0pt}}\nolimits}

\begin{document}

\title[Currents in Heisenberg groups
] 
{Currents in Heisenberg groups
}

\author[Bruno Franchi, Pierre Pansu]{
Bruno Franchi\\ Pierre Pansu
}

\begin{abstract}
There are three approaches to currents tuned to the anisotro\-pic geometry of Heisenberg groups: Ambrosio and Kirchheim's approach valid for general metric spaces; distributions dual to horizontal differential forms; distributions dual to Rumin's complex. It is shown that, in dimensions less than half the ambient dimension, these three theories coincide. On the other hand, they diverge beyond middle dimension: Ambrosio-Kirchheim currents vanish, Rumin currents correspond to a new class of Federer-Fleming currents called \emph{oblique currents}.
\end{abstract} 

\maketitle 

\tableofcontents

\section{Introduction}

\subsection{Currents}
The notion of currents in Euclidean space (or, more generally, in Riemannian manifolds) goes back to G. de Rham \cite{deRham}, H. Whitney \cite{whitney}, and H. Federer \& W.H. Fleming \cite{federer_fleming}. 
Roughly speaking, currents are distributions defined on the space of differential forms generalizing the notion of submanifolds to nonsmooth objects. 
Quoting \cite{federer_fleming}, ``the theory of currents is designed to permit a treatment of surfaces that may be very irregular, while retaining enough structure to support the familiar geometric operations.''  
For a general overview of the theory of currents, we refer e.g. to \cite{federer_fleming}, \cite{federer}, \cite{morgan}, \cite{GMS},\cite{simon}.

\subsection{Anisotropic geometries}

In 2000, L. Ambrosio \& B. Kirchheim \cite{AK_acta}, following ideas of E. De Giorgi \cite{degiorgi}, presented a theory of currents in complete metric spaces, including masses, normal currents, rectifiable and integral currents. Their rectifiable and integral currents are currents of integration on rectifiable sets modelled on Euclidean spaces. 
In \cite{AK_mathann}, they further investigated the rectifiability of sets in metric spaces and they focussed on the prominent example of Heisenberg groups endowed with dilation-homogeneous metrics.

Here is the context. The large scale geometry of the $2n+1$-dimensional Heisenberg group $\he n$ is governed by a one-parameter group $(\mathfrak s_\lambda)$ of automorphisms, which replace the homotheties of vectorspaces. Let us call them \emph{anisotropic dilations}. 
Left-invariant distances homogeneous under anisotropic dilations are efficient tools for understanding large scale geometry. For instance, their Hausdorff dimension $2n+2$ reflects the growth of the volume of large balls. Popular examples of anisotropic metrics are the Cygan-Kor\'anyi metric or the subRiemannian metric defined in Section \ref{rumin heisenberg}.

\subsection{Anisotropic currents}

Ambrosio \& Kirchheim discovered that the anisotropic $\he n$ contained no rectifiable sets in their sense, hence no rectifiable currents, of dimension $>n$. This is an indication that interesting phenomena might occur when comparing \emph{metric currents} with other notions of anisotropic currents inspired by differential geometry.

Differential forms on $\he{n}$ split into vertical and horizontal differential forms. As a consequence, classical Federer-Fleming currents split into \emph{horizontal} and \emph{vertical currents}, subspaces which behave well under anisotropic dilations (see Section \ref{rumin and FF}). Masses, normal currents and, in dimensions $\le n$, integral currents make sense in these subspaces. 
However, vertical currents are rare and do not lead to a satisfactory theory in dimensions $>n$.

To recover a full dilation-invariant chain complex of forms, M. Rumin, in \cite{rumin_jdg}, had to redesign differential forms. This led him to a substitute for de Rham's complex, that recovers scale invariance under $(\mathfrak s_\lambda)$. We refer to Section \ref{rumin heisenberg} for a list of the main properties of Rumin's complex, as well as to \cite{rumin_jdg} and \cite{BFTT} for details of the construction. By duality, Rumin's differential forms give rise to a notion of current in $\he n$. This construction appeared first in \cite{FSSC_advances} (see also \cite{BFTT}, \cite{V}). This theory of \emph{Rumin currents} has notions of masses, normal currents, and, in dimensions $\le n$, integral currents.

Rumin himself has suggested an alternative to horizontal/vertical currents, inspired by his embedding of Rumin's complex as a subcomplex of de Rham's complex. In this paper, these will be called Federer-Fleming \emph{oblique currents}. They form a subcomplex of the chain complex of Federer-Fleming currents. This theory has a notion of mass, called \emph{oblique mass}, which diverges from the usual mass in dimensions $>n$. On the other hand, in dimensions $\le n$, the theory coincides with horizontal currents. 

\subsection{Results}
So we have three possible approaches to anisotropic currents in Heisenberg groups. How are they related? 

Generalizing Ambrosio \& Kirchheim's result on rectifiable currents, M. Williams proved that all finite mass metric currents vanish in dimensions $>n$, \cite{Williams}. He suggested a possible connection between metric and Rumin currents in dimension $k\le n$. The following theorem confirms his intuition.

\begin{theorem} \label{main}
Let $\he n$ denote the $2n+1$-dimensional Heisenberg
group equipped with a subRiemannian or Koranyi-Cygan distance.
\begin{enumerate}
  \item In dimensions $k\le n$, there is a 1-1 correspondence between Rumin currents and Federer-Fleming horizontal currents. Ambrosio-Kirchheim metric currents embed in Federer-Fleming horizontal currents. Via this embedding, Ambrosio-Kirchheim metric currents of finite mass coincide with Federer-Fleming horizontal currents of finite mass. These correspondences are isomorphic (preserve masses up to a multiplicative constant) and commute with boundary operators.
  \item In dimensions $k\le n$, the above correspondences map integral currents to integral currents.
  \item In dimensions $k>n$, there is a 1-1 correspondence between Rumin currents and Federer-Fleming oblique currents. This correspondence commutes with boundary operators, and there is an expression for the oblique mass (Rumin mass expressed in Federer-Fleming terms).
\end{enumerate}
\end{theorem}
In other words, in dimensions $k\le n$, for currents of finite mass, the three theories coincide. In dimensions $k>n$, they tend to diverge.

\medskip

Theorem \ref{main} leaves many questions unanswered.
\begin{enumerate}
  \item Do their exist nonzero Ambrosio-Kirchheim metric currents in dimensions $k>n$ at all?
  \item Does the above correspondence between metric and Rumin currents preserve masses exactly? 
  \item A notion of flat current can be defined in each theory. Do the above correspondences extend to flat currents?
  \item Spaces of Rumin normal currents with bounded normal mass are flat compact (see \cite{Julia}). Therefore, a notion of Rumin integral current (flat limit of currents of integration on certain submanifolds with bounded Rumin normal mass) can be defined also in dimensions $k>n$ (as suggested in \cite{FSSC_advances}). Can this notion be expressed in Federer-Fleming terms? In particular, can these objects be expressed as currents of integration?
\end{enumerate}

\medskip

The following picture summarizes the bijections established between various kinds of currents.

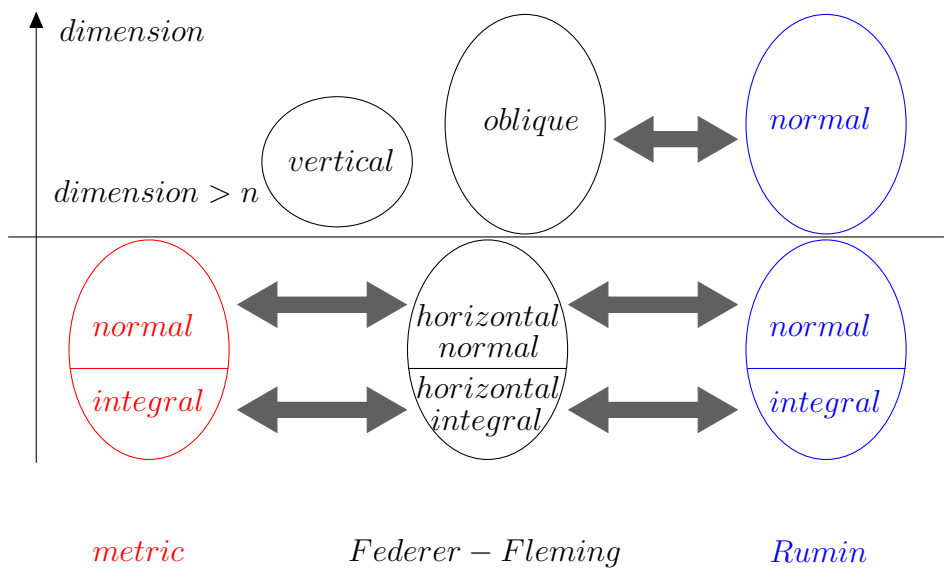
\begin{center}
\definecolor{wqwqwq}{rgb}{0.3764705882352941,0.3764705882352941,0.3764705882352941}
\definecolor{qqqqff}{rgb}{0.,0.,1.}
\definecolor{ffqqqq}{rgb}{1.,0.,0.}
\begin{tikzpicture}[line cap=round,line join=round,>=triangle 45,x=1.0cm,y=1.0cm]
\clip(-0.36886976647427416,-4.501912770291697) rectangle (12.063887539548935,3.2780175114691916);
\fill[color=wqwqwq,fill=wqwqwq,fill opacity=1.0] (3.,-2.2) -- (3.,-2.4) -- (4.6,-2.4) -- (4.6,-2.2) -- cycle;
\fill[color=wqwqwq,fill=wqwqwq,fill opacity=1.0] (3.2,-2.6) -- (3.2,-2.) -- (2.6803847577293367,-2.3) -- cycle;
\fill[color=wqwqwq,fill=wqwqwq,fill opacity=1.0] (4.4,-2.) -- (4.4,-2.6) -- (4.919615242270663,-2.3) -- cycle;
\fill[color=wqwqwq,fill=wqwqwq,fill opacity=1.0] (7.4,-2.2) -- (7.4,-2.4) -- (9.,-2.4) -- (9.,-2.2) -- cycle;
\fill[color=wqwqwq,fill=wqwqwq,fill opacity=1.0] (3.,-1.) -- (3.,-0.8) -- (4.6,-0.8) -- (4.6,-1.) -- cycle;
\fill[color=wqwqwq,fill=wqwqwq,fill opacity=1.0] (3.2,-1.2) -- (3.2,-0.6) -- (2.680384757729337,-0.9) -- cycle;
\fill[color=wqwqwq,fill=wqwqwq,fill opacity=1.0] (4.4,-0.6) -- (4.4,-1.2) -- (4.919615242270663,-0.9) -- cycle;
\fill[color=wqwqwq,fill=wqwqwq,fill opacity=1.0] (7.6,-2.6) -- (7.6,-2.) -- (7.080384757729337,-2.3) -- cycle;
\fill[color=wqwqwq,fill=wqwqwq,fill opacity=1.0] (8.8,-2.) -- (8.8,-2.6) -- (9.319615242270665,-2.3) -- cycle;
\fill[color=wqwqwq,fill=wqwqwq,fill opacity=1.0] (7.4,-1.) -- (7.4,-0.8) -- (9.,-0.8) -- (9.,-1.) -- cycle;
\fill[color=wqwqwq,fill=wqwqwq,fill opacity=1.0] (7.6,-1.2) -- (7.6,-0.6) -- (7.080384757729337,-0.9) -- cycle;
\fill[color=wqwqwq,fill=wqwqwq,fill opacity=1.0] (8.8,-0.6) -- (8.8,-1.2) -- (9.319615242270665,-0.9) -- cycle;
\fill[color=wqwqwq,fill=wqwqwq,fill opacity=1.0] (8.,1.4) -- (8.,1.2) -- (9.,1.2) -- (9.,1.4) -- cycle;
\fill[color=wqwqwq,fill=wqwqwq,fill opacity=1.0] (8.2,1.) -- (8.2,1.6) -- (7.6803847577293345,1.3) -- cycle;
\fill[color=wqwqwq,fill=wqwqwq,fill opacity=1.0] (8.8,1.6) -- (8.8,1.) -- (9.319615242270665,1.3) -- cycle;
\draw [domain=-0.36886976647427416:12.063887539548935] plot(\x,{(-0.-0.*\x)/12.});
\draw [->] (0.,-3.) -- (0.,3.);
\draw [rotate around={90.:(6.,-1.5)}] (6.,-1.5) ellipse (1.46040481324094cm and 1.0643224222655987cm);
\draw [rotate around={90.:(1.5,-1.5)},color=ffqqqq] (1.5,-1.5) ellipse (1.4604048132409453cm and 1.0643224222656027cm);
\draw [rotate around={90.:(10.5,-1.5)},color=qqqqff] (10.5,-1.5) ellipse (1.4604048132409528cm and 1.064322422265608cm);
\draw [rotate around={90.:(10.5,1.5)},color=qqqqff] (10.5,1.5) ellipse (1.4604048132409528cm and 1.064322422265608cm);
\draw [rotate around={90.:(6.5,1.5)}] (6.5,1.5) ellipse (1.4604048132409486cm and 1.0643224222656051cm);
\draw [rotate around={0.:(4.,1.)}] (4.,1.) ellipse (1.cm and 0.8660254037844386cm);
\draw [color=ffqqqq] (0.45138822059687417,-1.75)-- (2.548611779403126,-1.75);
\draw (4.9513882205968756,-1.75)-- (7.0486117794031244,-1.75);
\draw [color=qqqqff] (9.451388220596854,-1.75)-- (11.548611779403151,-1.75);
\draw [color=wqwqwq] (3.,-2.2)-- (3.,-2.4);
\draw [color=wqwqwq] (3.,-2.4)-- (4.6,-2.4);
\draw [color=wqwqwq] (4.6,-2.4)-- (4.6,-2.2);
\draw [color=wqwqwq] (4.6,-2.2)-- (3.,-2.2);
\draw [color=wqwqwq] (3.2,-2.6)-- (3.2,-2.);
\draw [color=wqwqwq] (3.2,-2.)-- (2.6803847577293367,-2.3);
\draw [color=wqwqwq] (2.6803847577293367,-2.3)-- (3.2,-2.6);
\draw [color=wqwqwq] (4.4,-2.)-- (4.4,-2.6);
\draw [color=wqwqwq] (4.4,-2.6)-- (4.919615242270663,-2.3);
\draw [color=wqwqwq] (4.919615242270663,-2.3)-- (4.4,-2.);
\draw [color=wqwqwq] (7.4,-2.2)-- (7.4,-2.4);
\draw [color=wqwqwq] (7.4,-2.4)-- (9.,-2.4);
\draw [color=wqwqwq] (9.,-2.4)-- (9.,-2.2);
\draw [color=wqwqwq] (9.,-2.2)-- (7.4,-2.2);
\draw [color=wqwqwq] (3.,-1.)-- (3.,-0.8);
\draw [color=wqwqwq] (3.,-0.8)-- (4.6,-0.8);
\draw [color=wqwqwq] (4.6,-0.8)-- (4.6,-1.);
\draw [color=wqwqwq] (4.6,-1.)-- (3.,-1.);
\draw [color=wqwqwq] (3.2,-1.2)-- (3.2,-0.6);
\draw [color=wqwqwq] (3.2,-0.6)-- (2.680384757729337,-0.9);
\draw [color=wqwqwq] (2.680384757729337,-0.9)-- (3.2,-1.2);
\draw [color=wqwqwq] (4.4,-0.6)-- (4.4,-1.2);
\draw [color=wqwqwq] (4.4,-1.2)-- (4.919615242270663,-0.9);
\draw [color=wqwqwq] (4.919615242270663,-0.9)-- (4.4,-0.6);
\draw [color=wqwqwq] (7.6,-2.6)-- (7.6,-2.);
\draw [color=wqwqwq] (7.6,-2.)-- (7.080384757729337,-2.3);
\draw [color=wqwqwq] (7.080384757729337,-2.3)-- (7.6,-2.6);
\draw [color=wqwqwq] (8.8,-2.)-- (8.8,-2.6);
\draw [color=wqwqwq] (8.8,-2.6)-- (9.319615242270665,-2.3);
\draw [color=wqwqwq] (9.319615242270665,-2.3)-- (8.8,-2.);
\draw [color=wqwqwq] (7.4,-1.)-- (7.4,-0.8);
\draw [color=wqwqwq] (7.4,-0.8)-- (9.,-0.8);
\draw [color=wqwqwq] (9.,-0.8)-- (9.,-1.);
\draw [color=wqwqwq] (9.,-1.)-- (7.4,-1.);
\draw [color=wqwqwq] (7.6,-1.2)-- (7.6,-0.6);
\draw [color=wqwqwq] (7.6,-0.6)-- (7.080384757729337,-0.9);
\draw [color=wqwqwq] (7.080384757729337,-0.9)-- (7.6,-1.2);
\draw [color=wqwqwq] (8.8,-0.6)-- (8.8,-1.2);
\draw [color=wqwqwq] (8.8,-1.2)-- (9.319615242270665,-0.9);
\draw [color=wqwqwq] (9.319615242270665,-0.9)-- (8.8,-0.6);
\draw [color=wqwqwq] (8.,1.4)-- (8.,1.2);
\draw [color=wqwqwq] (8.,1.2)-- (9.,1.2);
\draw [color=wqwqwq] (9.,1.2)-- (9.,1.4);
\draw [color=wqwqwq] (9.,1.4)-- (8.,1.4);
\draw [color=wqwqwq] (8.2,1.)-- (8.2,1.6);
\draw [color=wqwqwq] (8.2,1.6)-- (7.6803847577293345,1.3);
\draw [color=wqwqwq] (7.6803847577293345,1.3)-- (8.2,1.);
\draw [color=wqwqwq] (8.8,1.6)-- (8.8,1.);
\draw [color=wqwqwq] (8.8,1.)-- (9.319615242270665,1.3);
\draw [color=wqwqwq] (9.319615242270665,1.3)-- (8.8,1.6);
\draw [color=ffqqqq](0.6,-1.8791810059379455) node[anchor=north west] {$integral$};
\draw [color=ffqqqq](0.6,-0.8830473069766647) node[anchor=north west] {$normal$};
\draw (4.9,-0.7443451463618028) node[anchor=north west] {$horizontal$};
\draw (4.9,-1.7) node[anchor=north west] {$horizontal$};
\draw [color=qqqqff](9.6,-0.8830473069766647) node[anchor=north west] {$normal$};
\draw [color=qqqqff](9.6,-1.8917902932665693) node[anchor=north west] {$integral$};
\draw [color=qqqqff](9.6,1.878386617991949) node[anchor=north west] {$normal$};
\draw (5.8,1.8657773306633254) node[anchor=north west] {$oblique$};
\draw (3.2,1.3235779755325017) node[anchor=north west] {$vertical$};
\draw (0.16072030132793153,3.051050339553963) node[anchor=north west] {$dimension$};
\draw (0.08506457735618786,0.907471493687916) node[anchor=north west] {$dimension > n$};
\draw [color=ffqqqq](0.6,-3.8840576911891307) node[anchor=north west] {$metric$};
\draw (4.,-3.871448403860507) node[anchor=north west] {$Federer-Fleming$};
\draw [color=qqqqff](9.6,-3.8840576911891307) node[anchor=north west] {$Rumin$};
\draw (5.2,-1.1604516282063884) node[anchor=north west] {$normal$};
\draw (5.1,-2.1313667525104214) node[anchor=north west] {$integral$};
\end{tikzpicture}
\end{center}

\subsection{A new point of view}

On the Heisenberg groups, the anisotropic dilations split covectors, and hence, differential forms, into horizontal and vertical. Since Rumin forms in degrees $\le n$ are horizontal, one defines horizontal Federer-Fleming currents $T$ as those which vanish on vertical forms, as well as their boundary $\partial T$. Good try! They exist in abundance (every Legendrian submanifold relative to the left-invariant contact structure provides one), and they turn out to match Rumin currents in dimensions $\le n$, as is stated in Theorem \ref{main}. Integral horizontal currents constitute a sound framework for geometric measure theory of horizontal objects, i.e. up to dimension $n$.

Since Rumin forms in degrees $> n$ are vertical, one defines vertical Federer-Fleming currents $T$ as those which vanish on horizontal forms, as well as their boundary $\partial T$. Bad try! There are very few of them. In particular, none of them is compactly supported, see Proposition \ref{bad}.

Fortunately, Michel Rumin is here to pull us out of this ditch. In \cite{rumin_cras}, he embeds the complex bearing his name as a subcomplex $E$ of de Rham's complex, which is neither horizontal not vertical. Let us call it the \emph{oblique} subcomplex. In his celebrated lecture notes \cite{rumin_palermo}, he views currents as given by integration against differential forms $\alpha$ with distributional coefficients. Then he investigates what it means for a submanifold that its Poincar\'e dual (intersection) current be defined by an oblique form $\alpha\in E$. In low dimensions, he recovers Legendrian submanifolds. In high dimensions, he finds co-Legendrian submanifolds with co-Legendrian boundary. Recall that a submanifold is co-Legendrian if at each point, the intersection of its tangent space with the contact structure is co-isotropic, i.e. contains its symplectic-orthogonal complement. Thus \emph{oblique currents} exist in abundance. Theorem \ref{main} states that they match Rumin currents in dimensions $>n$. For currents of integration on $k$-dimensional submanifolds, the oblique mass coincides (up to densities) with the Hausdorff $k+1$-measure. We hope that oblique currents will constitute a useful addition to the toolbox of geometric measure theory in Heisenberg groups.

\subsection{Techniques}

The most delicate point is to show that Federer-Fleming horizontal integral currents are Ambrosio-Kirchheim metric integral currents. Indeed, there is strong evidence that horizontal rectifiable Federer-Fleming currents need not be rectifiable metric currents, see Example \ref{Cantor}. The proof, by approximation, uses a modification of Federer \& Fleming's Deformation Theorem, the random deformation theorem \ref{deformation}, \cite{contactfilling}, and a deformation  of chains of a triangulation to horizontal chains due to R. Young, \cite{YoungLowDim}.

\subsection{Structure of the paper}

Section \ref{preliminary} recalls properties of Rumin's complex. The definitions of Rumin currents, and of horizontal and oblique currents can be found in Section \ref{rumin and FF}. Section \ref{correspondences} establishes links between these various differential geometric notions. Section \ref{random} recalls the random deformation theorem. The comparison between metric and Federer-Fleming currents appears in Section \ref{AK versus FF}. This completes the proof of Theorem \ref{main}.

\section{Definitions and preliminary results}\label{preliminary}

For a general review on Heisenberg groups and their properties, we
refer to \cite{Stein}, \cite{GromovCC}, \cite{BLU}, and to \cite{VarSalCou}.
We limit ourselves to fix some notations.

\subsection{Heisenberg groups as Riemannian Lie groups}

We denote by $\he n$  the $2n+1$-dimensional Heisenberg
group, identified with $\rn {2n+1}$ through exponential
coordinates. A point $p\in \he n$ is denoted by
$p=(x,y,t)$, with both $x,y\in\rn{n}$
and $t\in\R$.
   If $p$ and $p'\in \he n$, the group operation is defined by

\begin{equation*}
p\cdot p'=(x+x', y+y', t+t' + \frac12 \sum_{j=1}^n(x_j y_{j}'- y_{j} x_{j}')).
\end{equation*}
The unit element of $\he n$ is the origin, that will be denoted by $e$.
Given $q\in\he n$, the \emph{left translation} $\tau_q:\he n\to\he n$ is defined
by $ p\mapsto\tau_q p:=q\cdot p$. 

The Lebesgue measure in $\mathbb R^{2n+1}$ 
is a Haar measure in $\he n$ (i.e., a bi-invariant measure on the group). 

We denote by $\mfrak h$ the Lie algebra of left
invariant vector fields of $\he n$. The standard basis of $\mfrak
h$ is given, for $i=1,\dots,n$,  by
\begin{equation*}
X_i := \partial_{x_i}-\frac12 y_i \partial_{t},\quad Y_i :=
\partial_{y_i}+\frac12 x_i \partial_{t},\quad Z :=
\partial_{t}.
\end{equation*}
The only non-trivial commutation  relations are $
[X_{i},Y_{i}] = Z $, for $i=1,\dots,n.$ 
The {\it horizontal subspace}  $\mfrak h_1$ is the subspace of
$\mfrak h$ spanned by $X_1,\dots,X_n$ and $Y_1,\dots,Y_n$:
${ \mfrak h_1:=\mathrm{span}\,\left\{X_1,\dots,X_n,Y_1,\dots,Y_n\right\}\,.}$

We refer to $X_1,\dots,X_n,Y_1,\dots,Y_n$
(identified with first order differential operators) as
the {\it horizontal derivatives}. Denoting  by $\mfrak h_2$ the linear span of $Z$, the $2$-step
stratification of $\mfrak h$ is expressed by
\begin{equation*}
\mfrak h=\mfrak h_1\oplus \mfrak h_2.
\end{equation*}

Deciding that the basis $X_1,\dots,X_n,Y_1,\dots,Y_n,Z$ is orthonormal defines a left-invariant Riemannian metric on $\he{n}$. This metric is used to define norms of differential forms. As for every Riemannian manifold, the Riemannian metric in each ball is equivalent to the Euclidean one. Therefore the theory of currents with bounded support in Euclidean space extends without effort to the Riemannian Heisenberg group.

\subsection{Heisenberg groups as nonRiemannian metric spaces}

The stratification of the Lie algebra $\mfrak h$ induces a family of anisotropic dilations 
$\mathfrak s_\lambda: \he n\to\he n$, $\lambda>0$ as follows: if
$p=(x,y,t)=(\bar p,p_{2n+1})\in \he n$, then
\begin{equation}\label{dilations}
\mathfrak s_\lambda (x,y,t) = (\lambda x, \lambda y, \lambda^2 t)= (\lambda \bar p, \lambda^2 p_{2n+1}).
\end{equation}

The Heisenberg group $\he n$ can be endowed with the \emph{Cygan-Kor\'anyi norm}
\begin{equation}\label{gauge}
\varrho (p)=\big(|\bar p|^4+ 16\, p_{2n+1}^2\big)^{1/4},
\end{equation}
and the corresponding left invariant distance 
$d(p,q):=\varrho ({p^{-1}\cdot q})$.

Alternatively, one can use the \emph{sub-Riemannian distance}, defined by minimizing the lengths of horizontal curves joining $p$ to $q$. Here, a piecewise $C^1$ curve $t\mapsto c(t)$ is \emph{horizontal} if for every $t$, the derivative $c'(t)$ left-translated to the origin belongs to the $(x,y)$ hyperplane. Its length is measured with respect to the left-invariant quadratic form which is equal to $|dx|^2+|dy|^2$ at the origin. We stress that the Cygan-Kor\'anyi distance is a true distance, see \cite{Stein}, p.\,638, which is equivalent to the sub-Riemannian distance.
 
The gauge norm \eqref{gauge} is $\mathfrak s_\lambda$-homogenous. So is the subRiemannian metric. It follows that the Lebesgue measure of the ball $B(x,r)$ is $r^{2n+2}$ up to a geometric constant
(the Lebesgue measure of $B(e,1)$). 
The constant 
$$
Q:=2n+2,
$$
is called the {\sl homogeneous dimension}  of $\he n$
with respect to $\mathfrak s_\lambda$, $\lambda>0$. It coincides with the Hausdorff dimension of $(\he n,d)$, which differs from its topological dimension, equal to $2n+1$.
    
\subsection{Differential forms}

The dual space of $\mfrak h$ is denoted by $\covH 1$.  The  basis of
$\covH 1$,  dual to  the basis $\{X_1,\dots , Y_n,Z\}$,  is the family of
covectors $\{dx_1,\dots, dx_{n},dy_1,\dots, dy_n$, $\theta\}$ where 
\begin{equation}\label{theta}
 \theta
:= dt - \frac12 \sum_{j=1}^n (x_jdy_j-y_jdx_j)
\end{equation}
 is called the {\it contact
form} in $\he n$. 
We denote by $\scalp{\cdot}{\cdot}{} $ the
inner product in $\covH 1$  that makes $(dx_1,\dots, dy_{n},\theta  )$ 
an orthonormal basis.

We set
\begin{equation*}
\omega_i:=dx_i, \quad \omega_{i+n}:= dy_i \quad { \mathrm{and} }\quad \omega_{2n+1}:= \theta, \quad \text
{for }i =1, \dots, n.
\end{equation*}

We put $\covH 0 =\R $
and, for $1\leq h \leq 2n+1$,
\begin{equation*}
\begin{split}
         \covH h& :=\mathrm {span}\{ \omega_{i_1}\wedge\dots \wedge \omega_{i_h}:
1\leq i_1< \dots< i_h\leq 2n+1\}.
\end{split}
\end{equation*}
The elements of $\covh{h}$ have similar expressions,
\begin{equation*}
\begin{split}
         \covh{h}& :=\mathrm {span}\{ \omega_{i_1}\wedge\dots \wedge \omega_{i_h}:
1\leq i_1< \dots< i_h\leq 2n\}.
\end{split}
\end{equation*}

Throughout this paper, the elements of $\cov h$ are identified with the left-invariant sections
of the vector bundle $\cov h$, i.e. with the \emph{left-invariant differential forms}
of degree $h$ on $\he n$. 

\begin{definition}\label{weight}
We say that a covector $\eta\neq 0$ has \emph{pure weight} $w$ if $\mathfrak{s}_{\lambda}^*\eta=\lambda^w \eta$.
\end{definition}

Then elements of $\covh 1$ have pure weight $1$ and $\theta$ has pure weight $2$. In degree $k$, covectors split orthogonally, 
\begin{align*}
\cov{k}=\covh{k}\oplus \theta\wedge\covh{k-1},
\end{align*}
where the first summand has pure weight $k$ and the second, $k+1$.

\begin{definition}
The \emph{horizontal $k$-covectors} are those which have pure weight $k$. The \emph{vertical $k$-covectors} are those which have pure weight $k+1$. 

A differential form is horizontal (resp. vertical) if all its left-translates, at the origin, are horizontal (resp. vertical).
\end{definition}

On left-invariant forms, the exterior differential preserves the weight. However, on nonleft-invariant differential forms, this is not the case anymore, this is the starting point of the construction of Rumin's complex.

\subsection{Rumin's complex in Heisenberg groups}\label{rumin heisenberg}

Let us give a short introduction to Rumin's complex, in the version . For a more detailed presentation we
refer to Rumin's papers \cite{rumin_cras},\cite{rumin_grenoble} and \cite{rumin_palermo}. Here we follow the presentation of \cite{BFTT}. Note that Rumin's initial version, \cite{rumin_jdg}, tailored for contact manifolds, was a bit different.

\subsubsection{Inverting the weight preserving differential}

The exterior differential $d$ does not preserve weights. It splits into
\begin{eqnarray*}
d=d_0+d_1+d_2
\end{eqnarray*}
where $d_0$ preserves weight, $d_1$ increases weight by 1 unit and $d_2$ increases weight by 2 units. 

It is crucial to notice that $d_0$ is an algebraic operator, in the sense that
for any real-valued $f\in\mc C^\infty (\he n)$ we have
$$
d_0(f\alpha)= f d_0\alpha,
$$
so that its action can be identified at any point with the action of a linear
operator from  $\cov h$ to $\cov {h+1}$ (that we denote again by $d_0$). 

We define now a (pseudo) inverse of $d_0$ as follows (see \cite{rumin_cras}, \cite{BFTT} Lemma 2.11): 

\begin{definition}\label{d_0}
If $\beta\in \cov {h+1}$, then there exists a unique $\alpha\in
\cov{h}\cap (\ker d_0)^\perp$ such that
$$
d_0\alpha-\beta\in \mc R(d_0)^\perp.
$$
We set $\alpha :=d_0^{-1}\beta$.
We notice that $d_0^{-1}$ preserves the weights. Since $d_0$ vanishes on horizontal forms, $d_0^{-1}$ takes its values in vertical forms.
\end{definition}

\subsubsection{The oblique subcomplex}

The next idea is to use $d_0^{-1}$ as a chain homotopy between the de Rham complex $(\Omega^\cdot,d)$ and a subcomplex. I.e., to set $\Pi:=Id-d_0^{-1}d-dd_0^{-1}$. It turns out that $\Pi$ is an idempotent, $\Pi\circ\Pi=\Pi$. Therefore its image $E:=\mc R(\Pi)$ is indeed a subcomplex. 
\begin{definition}
In the present paper, the subcomplex $E$ will be called the \emph{oblique subcomplex}.

To avoid confusions, one prefers to use the notation $\Pi_E$ instead of $\Pi$. 
\end{definition}
Since, by construction, $\Pi_E$ is chain homotopic to the identity, the cohomologies of $(E,d)$ and $(\Omega^\cdot,d)$ are isomorphic on all open subsets of $\he{n}$.

\subsubsection{The subbundle $E_0$}
The next task is to describe the oblique subcomplex as the space of smooth sections of a bundle. This will allow to equip it with norms, for instance. This is done in two steps. 

The first step is to view $\Pi_0:=Id-d_0^{-1}d-dd_0^{-1}$ as an approximation to $\Pi_E$. It is again an idempotent, and it is algebraic, hence its image is the space of smooth sections of a left-invariant subbundle $E_0$ of the vector bundle of differential forms. It is the left-invariant subbundle generated by a subspace of $\cov{\cdot}$, still denoted by $E_0$. As such, it inherits a Euclidean norm.
\begin{notation}
To avoid confusions, one prefers to use the notation $\Pi_{E_0}$ instead of $\Pi_0$. 
\end{notation}
By construction of $d_0^{-1}$, $\Pi_{E_0}$ is an orthogonal projector. Theorem \ref{main rumin new} expresses the fact that $\Pi_{E}:C^{\infty}(E_0)\to E$ and the restriction of $\Pi_{E_0}$ to $E$, from $E\to C^{\infty}(E_0)$, are bijections which are inverses of each other. This allows to view $(E,d)$ as an operator on $C^{\infty}(E_0)$.

\begin{definition}[M. Rumin, \cite{rumin_cras}]
The sections of the bundle $E_0$ are called \emph{Rumin forms}. The operator
 $$
 d_c:=\Pi_{E_0}\, d\,\Pi_{E}
 $$
on smooth Rumin forms is called \emph{the Rumin complex}.
\end{definition}

Its properties are summarized in the following

\begin{theorem}[\cite{rumin_cras}] \label{main rumin new}
The de Rham complex $(\Omega^\cdot,d)$ 
splits into the direct sum of two sub-complexes $(E^\cdot,d)$ and
$(F^\cdot,d)$, with
$$
E:=\ker d_0^{-1}\cap\ker (d_0^{-1}d)\quad\mbox{and}\quad
F:= \mc R(d_0^{-1})+\mc R (dd_0^{-1}).
$$
Let $\Pi_E$ be the projection on $E$ along $F$ (that
is not an orthogonal projection). We have

\begin{itemize}

\item[i)] $\Pi_{E}$ is a chain map, i.e. $d\Pi_{E} = \Pi_{E}d$.

\item[ii)]   If $\gamma\in C^{\infty}(E_0^{h})$,  then
\begin{itemize}
\item[$\bullet$] $
\Pi_E\gamma=\gamma -d_0^{-1}
d_1 \gamma$ if $1\le h\le n$ ;
\item[$\bullet$] $
\Pi_E\gamma=\gamma $ if $h>n$ .
\end{itemize}
In particular, $\Pi_E \gamma-\gamma$ is vertical.
\item[iii)] $\Pi_{E_0}$ preserves weights.

\item[iv)] $\Pi_{E_0}\Pi_{E}\Pi_{E_0}=\Pi_{E_0}$ and
$\Pi_{E}\Pi_{E_0}\Pi_{E}=\Pi_{E}$.

\item[v)] $d_c^2=0$.

\item[vi)] The complex $(C^{\infty}(E_0^\cdot),d_c)$ is locally exact.

\item[vii)] $d_c: C^{\infty}(E_0^h)\to C^{\infty}(E_0^{h+1})$ is a homogeneous differential operator in the 
horizontal derivatives
of order 1 if $h\neq n$, whereas $d_c: E_0^n\to E_0^{n+1}$ is a homogeneous differential operator in the 
horizontal derivatives
of order 2.
\end{itemize}
\end{theorem}

\subsubsection{Description of $E_0$} \label{E0E}

The second step consists in computing $E_0$ and $\Pi_{E_0}$, and find a handy characterization of $E$.

We need notation from symplectic linear algebra. 
\begin{notation}
Let $\mathfrak{h}_1=\R^{2n}$ be equipped with a symplectic form denoted by $d\theta$ and a compatible Euclidean structure. Let $L:\covh{\cdot}\to\covh{\cdot}$ denote the operator of wedge multiplication with $d\theta$. Let $\Lambda$ denote the adjoint of $L$.   
\end{notation}

\begin{proposition}[Formulae (58) and (61) in \cite{rumin_palermo}] \label{E0h}
If $0\le h\le n$, then
\begin{align*}
E_0^h&=\covh{h}\cap(\mc R(L))^\perp\\
&=\covh{h}\cap\ker(L^{n-k+1}).
\end{align*}
Therefore, if covectors are split as
\begin{align*}
\cov{h}=\theta\wedge\covh{h-1}\oplus (\covh{h}\cap \mc R(L))\oplus (\covh{h}\cap\ker(L^{n-k+1})),
\end{align*}
$\Pi_{E_0}$ is the projector onto the third summand. In particular, $\Pi_{E_0}$ kills every multiple of $\theta$ or of $d\theta$.

If $n+1\le h\le 2n+1$, then
\begin{align*}
E_0^h&=\theta\wedge(\covh{h-1}\cap\ker(L)).
\end{align*}
Therefore, if covectors are split as
\begin{align*}
\cov{h}=\covh{h}\oplus\theta\wedge\mc R(L^{k-n+1})\oplus\theta\wedge\ker(L),
\end{align*}
$\Pi_{E_0}$ is the projector onto the third summand. 
\end{proposition}

\subsubsection{Description of the oblique complex}

\begin{proposition}[Formula (62) in \cite{rumin_palermo}]\label{RuminE}

For each $0\le h\le 2n+1$, the subspace $E^h$ is defined in the space of smooth $h$-forms by the following equations.

If $0\le h\le n$, 
\begin{align*}
E^h&=\{\alpha\,;\,\Lambda(\alpha_H)=0,\,\Lambda((d\alpha)_H)=0\}\\
&=\{\alpha\,;\,\theta\wedge (d\theta)^{n-h+1}\wedge\alpha=0,\,\theta\wedge (d\theta)^{n-h}\wedge d\alpha=0\},
\end{align*}
(where $\alpha_H$ denotes the restriction of $\alpha$ to the contact hyperplane).

If $n+1\le h\le 2n+1$, 
\begin{align*}
E^h&=\{\alpha \text{ vertical and }d\alpha \text{ vertical}\}\\
&=\{\alpha=\theta\wedge\beta\,;\,d\theta\wedge\beta=0\}.
\end{align*}
\end{proposition}

In both Propositions \ref{E0h} and \ref{RuminE}, the occurrence of powers of $L$ stems from the following

\begin{lemma} \label{Hodge}

\begin{enumerate}
  \item Let $\star$ denote the Hodge operator on $\covh{\cdot}$. Then $\star\Lambda=L\star$. 
  \item In $\covh{k}$, $\ker(\Lambda)=\ker(L^{n-k+1})$.
  \item In $\covh{2n-k}$, $\mc R(\Lambda)=\mc R(L^{n-k+1})$.
  \item Let $\alpha\in\cov{h}$, $\alpha'\in\cov{2n+1-h}$ be covectors of complementary degrees. Then
\begin{align*}
(d_0^{-1}\alpha)\wedge\alpha'=(-1)^{h}\alpha\wedge d_0^{-1}\alpha'.
\end{align*}
\end{enumerate}
\end{lemma}

\begin{proof}
1. For covectors $\alpha$ and $\beta$ on $\mathfrak{h}_1$ of degrees $\ell,\ell'$, $\ell+\ell'=2n-2$, the equation
\begin{align*}
\star \Lambda(\alpha)\wedge\beta&=(\Lambda(\alpha)\cdot\beta)\star 1=(\alpha\cdot L\beta)\star 1=(\star\alpha)\wedge L\beta\\
&=(\star\alpha)\wedge d\theta\wedge\beta=L(\star\alpha)\wedge\beta
\end{align*}
implies that $\star \Lambda=L\star$. Therefore $\Lambda$ coincides with A. Weil's operator $\Lambda$ (\cite{weil}, paragraphe 4, page 21).
 
2. The Corollaire on page 28 of \cite{weil} states in particular that $\ker(\Lambda)=\ker(L^{n-k+1})$ in $\covh{k}$. 

3. Passing to orthogonal subspaces yields $\mc R(L)=\mc R(\Lambda^{n-k+1})$ in $\covh{k}$. Hence
\begin{align*}
\mc R(L^{n-k+1})&=\mc R(\star\Lambda^{n-k+1}\star^{-1})=\star\mc R(\Lambda^{n-k+1})\\
&=\star\mc R (L)=\mc R(\star^{-1}L\star)=\mc R(\Lambda)
\end{align*}
in $\covh{2n-k}$, since $\star\circ\star=\pm Id$.

4. For $\alpha\in\cov{h}$, let
\begin{align*}
\alpha=\alpha_V +L\beta +\gamma
\end{align*}
where $\alpha_V$ is vertical, $\beta$ and $\gamma$ are horizontal, and $\beta\in\ker(L)^\bot$, $\gamma\in\mc R(L)^\bot$. Then
\begin{align*}
d_0^{-1}\alpha=\theta\wedge\beta.
\end{align*}
Idem, let $\alpha=\alpha'_V +L\beta' +\gamma'$. Since $\beta\in\covh{h-1}\cap\ker(L)^\bot=\mc R(\Lambda)=\mc R(L^{n-h})$, there exists $\beta''$ such that $\beta=L^{n-h}\beta''$. 
Also, $\gamma'\in\covh{2n-(h-1)}\cap\mc R(L)^\bot=\ker\Lambda=\ker(L^{n-h})$, so
\begin{align*}
\beta\wedge\gamma'=(d\theta)^{n-h}\wedge\beta''\wedge\gamma'=\beta''\wedge(d\theta)^{n-h}\wedge\gamma'=\beta''\wedge L^{n-h}\gamma'=0.
\end{align*}
Idem, $\beta'\wedge\gamma=0$. It follows that
\begin{align*}
d_0^{-1}\alpha\wedge\alpha'
&=\theta\wedge\beta\wedge(L\beta'+\gamma')
=\theta\wedge \beta\wedge d\theta\wedge\beta'\\
&=\theta\wedge d\theta\wedge\beta\wedge\beta'
=\theta\wedge (L\beta+\gamma)\wedge\beta'\\
&=(-1)^{h}\alpha\wedge\theta\wedge \beta'
=(-1)^{h}\alpha\wedge d_0^{-1}\alpha'.
\end{align*}

\end{proof}

\section{Horizontal, oblique and Rumin Federer-Fleming currents}\label{rumin and FF}

\subsection{Horizontal currents}\label{defhoriz}

Throughout this section, $\mc U$ is an open subset of $\he n$.

Let us recall the following classical definitions (see \cite{federer}, Ch. IV).

\begin{definition}
We denote by $\mathcal D^{k}(\mathcal U)$ the space of all compactly
supported smooth $k$-forms on $\mc U$ endowed with its natural topology,
and by $\mathcal D_{k}(\mathcal U)$ its dual space (the space of Federer-Fleming $k$-currents
on $\mc U$).

If $T\in \mathcal D_{k}(\mathcal U)$ we define its (Riemannian) mass $\mc M(T)$
by
$$
\mc M(T):= \sup_{\omega\in \mathcal D^{k}(\mathcal U), \|\omega\|_{\infty}\le 1}
\Scal{T}{\omega}.
$$
\end{definition}

Recall the splitting of covectors
\begin{equation}\label{dec-weights}
\covH h = \covw {h}{h}\oplus \covw {h}{h+1} =  \covh h\oplus \Big(\covh {h-1}\Big)\wedge \theta.
\end{equation}
Sections of the first summand are called \emph{horizontal} differential forms, and sections of the second summand \emph{vertical}. Thus a differential form is vertical if it is divisible by the contact form $\theta$. 

\begin{definition}
A Federer-Fleming current $T$ is \emph{horizontal} if $T$ and $\partial T$ vanish on vertical forms. Equivalently, if
$$
T\res\theta=0,\quad T\res d\theta=0.
$$
\end{definition}
The equivalent definition stems from the following formula
\begin{align*}
(\partial T)\res\phi=T\res(d\phi)+(-1)^{\text{degree}(\phi)}\partial (T\res\phi),
\end{align*}
which follows from the pointwise identity
\begin{align*}
d(\phi\wedge\omega)=(d\phi)\wedge\omega+(-1)^{\text{degree}(\phi)}\phi\wedge(d\omega).
\end{align*}

Horizontal submanifolds provide a wealth of examples of normal horizontal Federer-Fleming currents. The subspace of horizontal Federer- Fleming currents is invariant under contactomorphisms.

\subsection{Vertical currents}

The naive definition
\begin{definition}[Useless definition] \label{useless}
A Federer-Fleming current $T$ is \emph{vertical} if $T$ and $\partial T$ vanish on horizontal forms.
\end{definition}
turns out to be too restrictive, for the following reason.

\begin{proposition}\label{bad}
If $T$ is a vertical current in the sense of Definition \ref{useless}, then $T\res\theta$ is invariant under vertical translations. Therefore no compactly supported Federer-Fleming current can be vertical.
\end{proposition}

\begin{proof}
Let $Z=\partial_t$ denote the Reeb vectorfield of $\theta$. Let $\mathcal{L}_Z$ denote the Lie derivative along $Z$. The vertical part of the exterior differential of a horizontal form $\alpha$ is $\theta\wedge\mathcal{L}_Z \alpha$. If $T$ is a vertical current, then for every horizontal form $\alpha$,
\begin{align*}
0=\Scal{\partial T}{\alpha}=\Scal{T}{d\alpha}=\Scal{T}{\theta\wedge \mathcal{L}_Z\alpha}=\Scal{T\res\theta}{\mathcal{L}_Z\alpha}=\Scal{\mathcal{L}_Z(T\res\theta)}{\alpha}.
\end{align*}
Since $\mathcal{L}_Z \theta=0$, 
\begin{align*}
\mathcal{L}_Z(T\res\theta)=\mathcal{L}_Z(T)\res\theta,
\end{align*}
which vanishes on vertical forms. It follows that $\Scal{\mathcal{L}_Z(T\res\theta)}{\omega}=0$ for all test forms $\omega$, i.e. $\mathcal{L}_Z(T\res\theta)=0$. This means that $T\res\theta$ is invariant under the flow generated by $Z$, which consists of vertical translations.

If $T$ is compactly supported, one finds that $T\res\theta=0$, $T$ vanishes both on horizontal and vertical forms, hence $T=0$.

\end{proof}

\subsection{Oblique currents}

Therefore, one must follow a different route. We use Section 5.3 in \cite{rumin_palermo} as a guideline and a source of examples. M. Rumin views currents as given by integration against differential forms $\alpha$ with distributional coefficients.

\begin{notation}\label{notFF}
Let $\mathcal{U}\subset\he{n}$ be an open set. For a differential form $\alpha$ with distributional coefficients on $\mathcal{U}$, let $\FF(\alpha)$ denote the Federer-Fleming current defined by
\begin{align*}
\Scal{\FF(\alpha)}{\omega}:=\int_{\mathcal{U}}\alpha\wedge\omega.
\end{align*}
\end{notation}

M. Rumin determines which submanifolds have the property that their Poincar\'e dual current belongs (in a generalized sense) to the subcomplex $E$ introduced in Theorem \ref{main rumin new}. These are the so-called \emph{co-Legendrian submanifolds} (his computation is reproduced in Example \ref{excolegendrian} below). A submanifold is co-Legendrian if at each point, the intersection of its tangent space with $\ker(\theta)$ is co-isotropic, i.e. contains its $d\theta$-orthogonal complement. Therefore co-Legendrian submanifolds with co-Legendrian boundary provide a wealth of examples.

This suggests defining \emph{smooth oblique currents} by Notation (\ref{notFF}), requiring that $\alpha\in E$. Next, we unravel the definition until it leads us to a formulation allowing data $\alpha$ with merely distributional coefficients.

We shall need the following integration by parts formula.

\begin{lemma}\label{dFF}
Let $0\le k \le 2n+1$. Let $\alpha$ be a smooth $2n+1-k$-form on an open set $\mathcal{U}$. Then
\begin{align*}
\partial\FF(\alpha)=(-1)^{k}\FF(d\alpha).
\end{align*}
\end{lemma}

\subsubsection{Low dimensional smooth oblique currents}

\begin{lemma}
Let $0\le k\le n$ and $h=2n+1-k$. Let $\alpha$ be a smooth $h$-form on an open set $\mc U$. Then 
\begin{eqnarray*}
\alpha\in E\iff \FF(\alpha) \text{ is horizontal}.
\end{eqnarray*}
\end{lemma}

\begin{proof}
Since $h\ge n+1$, Lemma \ref{RuminE} yields
\begin{align*}
\alpha\in E &\iff \alpha \text{ and }d\alpha \text{ are vertical}\\
&\iff \FF(\alpha) \text{ and }\partial \FF(\alpha) \text{ vanish on vertical forms}\\
&\iff \FF(\alpha) \text{ is horizontal}.
\end{align*}
\end{proof}

Therefore, smooth oblique currents in dimensions $\le n$ are nothing but smooth horizontal currents. There is no point in commenting further general oblique currents in dimensions $\le n$, they coincide with the horizontal Federer-Fleming currents of Section \ref{defhoriz}.

\subsubsection{High dimensional smooth oblique currents}

\begin{lemma}\label{check}
Let $n+1\le k \le 2n+1$ and $h=2n+1-k$. Let $\alpha$ be a smooth $h$-form on an open set $\mathcal{U}$. Then
\begin{align*}
\alpha\in E \iff \FF(\alpha)\res(\theta\wedge (d\theta)^{n-h+1})=0 \text{ and }\partial\FF(\alpha)\res (\theta\wedge (d\theta)^{n-h})=0.
\end{align*}
\end{lemma}

\begin{proof}
According to Proposition \ref{RuminE}, and using Lemma \ref{dFF},
\begin{align*}
\alpha\in E^h &\iff \theta\wedge (d\theta)^{n-h+1}\wedge\alpha=0 \text{ and }\theta\wedge (d\theta)^{n-h}\wedge d\alpha=0\\
&\iff \forall\omega\in\mc D^k(\mc U),\quad\int_{\mathcal{U}}\theta\wedge (d\theta)^{n-h+1}\wedge\alpha\wedge\omega=0 \\
&\hskip1.2cm\text{ and }\int_{\mathcal{U}}\theta\wedge (d\theta)^{n-h}\wedge d\alpha\wedge\omega=0\\
&\iff \FF(\alpha)\res\theta\wedge (d\theta)^{n-h+1}=0 \text{ and }\FF(d\alpha)\res \theta\wedge (d\theta)^{n-h}=0\\
&\iff \FF(\alpha)\res\theta\wedge (d\theta)^{n-h+1}=0 \text{ and }\partial\FF(\alpha)\res \theta\wedge (d\theta)^{n-h}=0.
\end{align*}
\end{proof}

This suggests the following definitions.

\begin{definition}
Let $n+1\le k \le 2n+1$. A $k$-dimensional Federer-Fleming current $\TFF$ is \emph{co-Legendrian} if
\begin{align*}
\TFF\res (\theta\wedge (d\theta)^{k-n})=0.
\end{align*}
\end{definition}

\begin{definition}
Let $n+1\le k \le 2n+1$. A $k$-dimensional Federer-Fleming current $\TFF$ is \emph{oblique} if $\TFF$ and $\partial \TFF$ are co-Legendrian, i.e. if
\begin{align*}
\TFF\res (\theta\wedge (d\theta)^{k-n})=0 \quad \text{and}\quad \partial \TFF \res (\theta\wedge (d\theta)^{k-n-1})=0.
\end{align*}
\end{definition}

\begin{remark}
The boundary of an oblique Federer-Fleming current of dimension $n+1$ is a horizontal Federer-Fleming current of dimension $n$. This raises the following filling problem. Does every closed horizontal integral $n$-current $S$ in $\he n$ bound an oblique current $T$ with oblique mass $\le O(\mc M(S)^{(n+2)/n})$, with $T$ integral in a suitable sense?
\end{remark}

\begin{example}[see Section 5.3 in \cite{rumin_palermo}] \label{excolegendrian}
Let $k=n+1,\ldots,2n+1$, let $h=2n+1-k$. Let $f:\mc U\to \R^h$ be a smooth map whose differential restricted to the contact hyperplane is onto. Let $\Phi:\R^h\to\R$ be a smooth function. Consider the averaged current of integration defined on $\mc D_k(\mc U)$ by
\begin{align*}
T_{f,\Phi}:\omega \mapsto \int_{\R^h}\left(\int_{f^{-1}(y)}\omega\right)\,\Phi(y)\,dy.
\end{align*}
Then 
 $T_{f,\Phi}$ is oblique if and only if the level sets $f^{-1}(y)$ are co-Legendrian submanifolds.
\end{example}

Letting $\Phi$ approach a Dirac mass, one gets that a smooth submanifold without boundary, transverse to the contact structure is co-Legendrian if and only if its current of integration is oblique. 

We shall see later, in Example \ref{exoblique}, that from a non-co-Legendrian submanifold, one recovers an oblique current by adding the boundary of a current with the same support, which vanishes on vertical forms.

\begin{proof}

Let $\Omega$ denote the volume form of $\R^h$. Then $T_{f,\Phi}=\FF(\alpha)$ for 
$\alpha=f^*\Phi\Omega$. Note that $d\alpha=0$. Remember the notation of paragraph \ref{E0E}: $L$ denotes multiplication with $d\theta$, $\Lambda$ the adjoint of $L$ and $\star$ the Hodge operator on covectors on the contact hyperplane $\ker(\theta)$. According to Proposition \ref{RuminE}, denoting by $\alpha_H$ the restriction of $\alpha$ to the contact hyperplane,
\begin{eqnarray*}
\alpha\in E \iff \alpha_H\in \mathcal{R}(L)^{\bot}.
\end{eqnarray*}
Lemma \ref{Hodge} implies that $\mathcal{R}(L)^{\bot}=\ker(\Lambda)=\ker(L\star)$. So
\begin{align*}
T_{f,\Phi} \text{ is oblique}\iff L\star\alpha_H=0.
\end{align*}
Let $\tau=\ker(df)$ denote the tangent space to the level sets $f^{-1}(y)$, $\tau_H:=\tau\cap\ker(\theta)$ its horizontal part and $\tau_H^\bot$ its orthogonal in $\ker(\theta)$. Since $\star\alpha_H$ is collinear to the simple covector associated to $\tau_H$, 
\begin{align*}
v\in\tau_H^\bot \iff \iota_v(\star\alpha_H)=0.
\end{align*}
For $v\in \tau_H^\bot$,
\begin{align*}
\iota_v(d\theta\wedge(\star\alpha_H))=\iota_v(d\theta)\wedge(\star\alpha_H)+d\theta\wedge\iota_v(\star\alpha_H)=\iota_v(d\theta)\wedge(\star\alpha_H).
\end{align*}
For $v,w\in \tau_H^\bot$,
\begin{align*}
\iota_w \iota_v(d\theta\wedge(\star\alpha_H))=d\theta(v,w)\star\alpha_H.
\end{align*} 
Thus 
\begin{align*}
L(\star\alpha_H)=0 \implies (d\theta)(v,w)=0 \text{ for all }v,w\in\tau_H^\bot \implies d\theta_{|\tau_H^\bot}=0,
\end{align*}
i.e., $\tau_H^\bot$ is isotropic. Let $J$ denote the almost complex structure on $\ker(\theta)$ relating the inner product to the symplectic form $d\theta$: $d\theta(v,Jw)=v\cdot w$. The $d\theta$-orthogonal of $\tau_H$ is $J\tau_H^\bot$. Since $d\theta(Jv,Jw)=d\theta(v,w)$, we conclude that the $d\theta$-orthogonal of $\tau_H$ is isotropic, i.e. $\tau_H$ is co-isotropic. Therefore, the level-sets $f^{-1}(y)$ are co-Legendrian.

Conversely, if all level-sets $f^{-1}(y)$ are co-Legendrian, $\tau_H$ is co-isotropic, and $\tau_H^\bot$ is isotropic. Pick an orthonormal basis $(e_i)$ of $\ker\theta$ such that $e_1,\ldots e_h$ is a basis of $\tau_H^\bot$. Let $\lambda_{1},\ldots, \lambda_{2n}$ be the dual basis. Then $\star\alpha_H$ is collinear to $\lambda_{h+1}\wedge\cdots \wedge \lambda_{2n}$. If
$$
d\theta=\sum_{i_1<i_2} t_{i_1,i_2}\lambda_{i_1}\wedge \lambda_{i_2},
$$
then, up to a scalar,
\begin{align*}
d\theta\wedge\star\alpha_H=\sum_{i_1<i_2\le h} t_{i_1,i_2}\lambda_{i_1}\wedge \lambda_{i_2}\wedge\lambda_{h+1}\wedge\cdots \wedge \lambda_{2n}.
\end{align*}
Since $d\theta$ vanishes on $\tau_H^\bot$, all $t_{i_1,i_2}$, $i_1<i_2\le h$ vanish, so $d\theta\wedge\star\alpha_H=0$. This shows that $T_{f,\Phi}$ is oblique.

\end{proof}

\subsection{Rumin currents}

\begin{notation} 
The space of smooth, compactly supported \Rumin forms on $\mc U$ is denoted by $\mathcal D(\mathcal U,E_0^k)$. 
\end{notation}

\begin{definition}\label{defcorrenti} 
Let $\mc U\subset \he n$ be
an open set.
We call \emph{\Rumin $k$-current}, $0\leq k \leq 2n+1$,
any continuous linear functional on $ \mathcal D(\mathcal U,E_0^k)$,
  and we denote by $\mathcal D_{\he{},k}(\mathcal U)$ 
  the set of all \Rumin $k$-currents.
     
  \end{definition}
  
    \begin{remark}
The definition of Rumin's current given in \cite{FSSC_advances} relies on the initial definition of Rumin's forms in \cite{rumin_jdg}, alluded to in Section \ref{rumin heisenberg}, involving quotients of spaces of differential forms. Clearly, the two classes of currents are isomorphic.
  \end{remark}
  
  \begin{definition}
  If  $\TR\in \mathcal D_{\he{},k}(\mathcal U)$, we define its (Rumin) \emph{mass} $\mc M_{\mathbb H}(\TR)$ 
by
$$
\mc M_{\mathbb H}(\TR):= \sup_{\gamma\in \mathcal D(\mathcal U, E_0^k),\, \|\gamma\|_{\infty}\le 1}
\Scal{\TR}{\gamma}.
$$

\end{definition}
  
\begin{definition}
 If $\TR\in \mathcal D_{\he{},k}(\mathcal U)$, we define its \emph{Rumin boundary}
 $\partial_{\mathbb H}\TR$
 by
 $$
 \Scal{\partial_{\mathbb H}\TR}{\gamma} := \Scal{\TR}{d_c\gamma} .
 $$
\end{definition}

\subsection{Integral currents in low dimensions}

The following result is a special instance of Pansu--Rademacher's differentiability theorem
in Carnot groups (see \cite{pansu_annals}).

\begin{proposition}\label{pansu rad}Let $\mc V\subset \R^k$ be an open set,
and   $\phi: \mc V\to \he n$ a Lipschitz map with respect to dilation homogeneous distances. 
Then $\phi$ is P-differentiable at a.e. $p_0\in \mc V$, i.e.
there is a (unique) graded group homomorphism $d_P f(p_0): \R^k \to \he n$ such that
\begin{equation}\label{P diff}
d_Pf_{p_0}(v ):=\lim_{s\to 0}\mathfrak s_{1/s}\left( \phi(p_0)^{-1}\phi(p_0+sv)\right)
\end{equation}
uniformly for $v$ in compact subsets of $\R^k$. 
\end{proposition}

\begin{lemma}\label{dtheta}
Let $\phi:\R^k\to\he{n}$ be a Lipschitz map. Then  $\phi^*\theta=0$, $\phi^*d\theta=0$ a.e. in $\R^k$. 
Furthermore, if $k>n$, then $\phi^*\omega=0$ a.e. for every differential $k$-form $\omega$.
\end{lemma}

\begin{proof} By Proposition \ref{pansu rad}, $\phi$ is P-differentiable as well as differentiable in the usual sense at a.e. $x\in \R^k$.
Without loss of generality we may assume $x=0$ and $\phi (0)=e$.

If we decompose $\phi$ according to the components of the different layers of $\he n$,
we can write $\phi=(\phi_1,\phi_2)$, where $\phi_i\in \mathfrak h_i$, $i=1,2$. By \eqref{P diff},
if $v\in \R^k$, then
$$
d_P\phi(0)(v) = \lim_{s\to 0} (s^{-1}\phi_1(sv), s^{-2} \phi_2(sv)),
$$
and hence
$$
\lim_{s\to 0} s^{-2} \phi_2(sv)\quad\text{exists},
$$
so that $\lim_{s\to 0} s^{-1} \phi_2(sv)=0$. By our choice of $\phi(0) =e$, $\theta(e)=dt$, so that
$$
\Scal{\phi^*\theta}{v} = \Scal{dt}{d\phi(0)(v)} = \lim_{s\to 0} s^{-1} \phi_2(sv)=0.
$$
Also 
$$
\Scal{\phi^*d\theta}{v} = \Scal{d\theta}{d\phi(0)(v)} = \Scal{d\theta}{\lim_{s\to 0} s^{-1} \phi_1(sv)}=\Scal{d\theta}{d_P\phi(0)}.
$$
Since $d_P\phi$ is a graded group morphism, its image is contained in a commutative and horizontal subgroup of $\he{n}$, on which the symplectic form $d\theta$ vanishes, so $\Scal{d\theta}{d_P\phi(0)}=0$ and $\phi^*d\theta(0)=0$.

If $k>n$, given  a $k$-form $\omega=\omega_H + \theta\wedge\omega_T$
with $\omega_H $ a horizontal $k$-form,  it follows from $\phi^*\theta=0$
that $\phi^*\omega = \phi^*\omega_H$. On the other hand, by a classical result in K\"ahlerian geometry
(see e.g. \cite{weil}, Theorem 3 p. 26), $\omega_H$ is divisible by $d\theta$, and then 
$\phi^*\omega_H =0 $ since $\phi^*d\theta=0$. 

\end{proof}

Next we define \Rumin integral currents, following Remark 5.7 and Definition 5.18 in \cite{FSSC_advances}, but allowing Lipschitz parametrizations (with respect to a dilation homogeneous metric) instead of $C^1$ maps.

\begin{definition} \label{defLipRumin}
A current $\TR\in \mathcal D_{\he{},k}(\mathcal U)$ of dimension $k$ is said a \emph{Lipschitz
current of integration} if there exist a compact subset $K$ of $\R^k$ and a Lipschitz map $\phi:\R^k \to \he n$ such that 
$$
\Scal{\TR}{\alpha}:= \int_K \phi^*\alpha \,d \mathcal L^k=\Scal{\phi_\#|\![K]\!|}{\alpha},
$$
where $\phi^*\alpha$ is defined a.e. thanks to Lemma \ref{dtheta}, $\mathcal L^k$ is the $k$-dimensional Lebesgue measure, and
$|\![K]\!|$ is the current in $\R^k$ defined by
$$
|\![K]\!|  = \mathcal L^k\res K.
$$

\end{definition}

\begin{remark} If $k>n$, by Lemma \ref{dtheta}, all Lipschitz currents of integration in $ \mathcal D_{\he{},k}(\mathcal U)$ are trivial.
On the contrary, it is shown in \cite{FSSC_advances} that there exist (nonLipschitz) currents of integration
in $ \mathcal D_{\he{},k}(\mathcal U)$ also when $k>n$.

\end{remark}

\begin{remark} 
Requiring parametrizations $\phi$ to be defined only on the subset $K$ instead on all of $\R^k$ might sound more appropriate. This makes no difference, since a Lipschitz extension theorem holds, see \cite{GromovCC} and \cite{WengerYoung}. No such extension property can hold if $k>n$, see \cite{BaloghFassler}.

\end{remark}

\begin{definition}   \label{defintRumin}
A \emph{Rumin integral current} is a current $\TR\in \mathcal D_{\he{},k}(\mathcal U)$ such that
 \begin{enumerate}
  \item $\TR$ is normal, i.e. $\mc M_{\mathbb H}(\TR)+\mc M_{\mathbb H}(\partial_{\mathbb H} \TR) $ is finite;
  \item $\TR$ is a sum of a countable family of Lipschitz currents of integration.\end{enumerate}

\end{definition}

\begin{remark}
The fact that the boundary of an integral current is again an integral current is not built in Definition \ref{defintRumin}. It will follow from Corollary \ref{integral}.
\end{remark}

\section{Rumin versus horizontal/oblique}\label{correspondences}

\subsection{Low dimensions}

\begin{proposition}\label{identificationcurrents}
If $1\leq k \leq n$, any Rumin current $\TR\in \mathcal D_{\he{},k}(\mathcal U)$
can be identified with a horizontal Federer-Fleming $k$-current $\widetilde \TR\in \mathcal D_{k}(\mathcal U)$,  by setting
\begin{equation*}
\Scal{\widetilde \TR}{\omega}:= \Scal{\TR}{\Pi_{E_0}\omega}, \qquad \text{for all }\omega \in \mathcal D^{k}(\mathcal U).
\end{equation*}
Moreover, $\mc M(\widetilde \TR)\le \mc M_{\mathbb H}(\TR)$ and $\partial\widetilde \TR=\widetilde{\partial_{\he{}}\TR}$.

Conversely, if $\TFF \in\mathcal D_{k}(\mathcal U)$ is a horizontal Federer-Fleming current, then $\TFF $ induces a \Rumin $k$-current
$\widehat \TFF \in \mathcal D_{\he{},k}(\mathcal U)$ setting
$$
\Scal{\widehat \TFF }{\gamma}:=\Scal{\TFF }{\Pi_E\gamma}\qquad\text{for all }\gamma\in \mc D(\mathcal{U},E_0^k).
$$
 Then $\mc M_{\mathbb H}(\widehat \TFF )\le  \mc M(\TFF )$ and $\partial_{\he{}}\widehat \TFF =\widehat{\partial \TFF }$.

\end{proposition}

The Rumin current $\widehat \TFF$ is well defined since $\Pi_E$, being a differential operator, is a continuous map from $ \mc D(\mathcal{U},E_0^k)$ to $\mc D(\mathcal{U},\cov{k})$.

\begin{proof} 
First of all, we notice that, if $\omega\in  \mathcal D^{k}(\mathcal U)$  
then $\Pi_{E_0}\omega\in \mathcal D(\mathcal U,  E_0^k)$, so that $\Scal{\widetilde \TR}{\omega} $
is well defined. In addition
\begin{equation}\label{aug31 eq:1}
\Scal{\widetilde \TR\res \theta}{\omega} = \Scal{\widetilde \TR}{\theta\wedge\omega} = \Scal{\TR}{\Pi_{E_0}(\theta\wedge\omega)} = 0,
\end{equation}
by Proposition \ref{E0h}. Moreover
\begin{equation}\label{aug31 eq:2}
\Scal{\partial \widetilde \TR\res \theta}{\omega} = \Scal{\widetilde \TR}{d (\theta\wedge\omega)}
=  \Scal{\TR}{\Pi_{E_0}(-\theta\wedge d\omega+d\theta\wedge\omega)} =0,
\end{equation}
by Proposition \ref{E0h} again. Finally, if $\omega\in \mathcal D^{h}(\mathcal U)$
and $\|\omega\|\le 1$, then $\|\Pi_{E_0}\omega\|\le 1$, so that
$$
\Scal{\widetilde \TR}{\omega} = \Scal{\TR}{\Pi_{E_0}\omega} \le \sup_{\xi\in \mc D(\mathcal{U},E_0^k),\, \|\xi\| \le 1}\Scal{\TR}{\xi} =\mc M_{\mathbb H}(\TR).
$$

Since $\partial\widetilde \TR$ is horizontal, for every $\omega\in\mathcal D^{k-1}(\mathcal U)$,
\begin{align*}
\Scal{\partial\widetilde \TR}{\omega}&=\Scal{\partial\widetilde \TR}{\Pi_{E}\Pi_{E_0}\omega}=\Scal{\widetilde \TR}{d\Pi_E\Pi_{E_0}\omega}=\Scal{\TR}{\Pi_{E_0}d\,\Pi_{E}\Pi_{E_0}\omega}\\
&=\Scal{\TR}{d_c\Pi_{E_0}\omega}=\Scal{\partial_{\he{}}\TR}{\Pi_{E_0}\omega}
=\Scal{\widetilde{\partial_{\he{}}\TR}}{\omega}.
\end{align*}

Take now $\gamma\in \mc D(\mc U,E_0^h)$. We have
$$
\Pi_E\gamma = \Pi_{E_0}\Pi_E\gamma + \Pi_{E_0}^\perp\Pi_E\gamma.
$$
For the first summand,
$$
\Pi_{E_0}\Pi_E\gamma = \Pi_{E_0}\Pi_E \Pi_{E_0}\gamma =  \Pi_{E_0}\gamma=\gamma.
$$
On the other hand, by Proposition \ref{E0h}, $\Pi_{E_0}^\perp\Pi_E\gamma= \xi_1 +  d\theta\wedge \xi_2$,
with $\xi_1\in (\covh{h})^\perp$, i.e. with $\xi_1=\theta\wedge\eta_1$.
Thus, for a horizontal Federer-Fleming current $\TFF $,
\begin{equation}\label{sept18 eq:1}
\Scal{\TFF }{\Pi_E\gamma} = \Scal{\TFF }{\gamma} + \Scal{\TFF \res \theta}{\eta_1} + \Scal{\TFF \res d \theta}{ \xi_2} =  \Scal{\TFF }{\gamma}.
\end{equation}

In order to estimate $\mc M_{\mathbb H} (\widehat \TFF )$, let us take $\gamma\in \mc D (\mc U, E_0^k)$ with 
$$\|\gamma\|_{L^\infty(\mc U,\cov{k})} = \|\gamma\|_{L^\infty(\mc U, E_0^k)}\le 1.$$
Then
$$
\Scal{\widehat \TFF }{\gamma}=\Scal{\TFF }{\Pi_E\gamma} = \Scal{\TFF }{\gamma}  \le  \mc M(\TFF ),
$$
so that $\mc M_{\mathbb H}(\widehat \TFF ) \le  \mc M(\TFF )$.

Since $E$ is a subcomplex, $dE\subset E$, so $(\Pi_{E}\Pi_{E_0})_{|dE}=(\Pi_{E}\Pi_{E_0}\Pi_{E})_{|dE}=(\Pi_{E})_{|dE}$ is the identity on $dE$. Therefore, for all $\gamma\in\mathcal D^{k-1}(\mathcal U, E_0^h)$,
\begin{align*}
\Scal{\partial_{\he{}}\widehat \TFF }{\gamma}
&=\Scal{\widehat \TFF }{d_c\gamma}=\Scal{\TFF }{\Pi_{E}\Pi_{E_0}\, d\,\Pi_{E}\gamma}=\Scal{\TFF }{d\,\Pi_{E}\gamma}\\
&=\Scal{\partial \TFF }{\Pi_{E}\gamma}=\Scal{\widehat{\partial \TFF }}{\gamma}.
\end{align*}

\end{proof}

\begin{proposition}\label{involutive}  
If $0\le k\le n$ and $\TR\in\mathcal D_{\he{},k}(\mathcal U)$, then
$$
\widehat {\widetilde \TR} =\TR.
$$
Conversely, if $\TFF \in \mathcal D_{k}(\mathcal U)$ is horizontal, then
$$
\widetilde{\widehat{\TFF }} = \TFF .
$$

\end{proposition}

\begin{proof} 

If $\gamma\in \mathcal D(\mathcal U,E_0^k)$, $\Pi_{E_0}\gamma=\gamma$, so, using Theorem \ref{main rumin new} iv),
\begin{align*}
\Scal{\widehat {\widetilde \TR} }{\gamma}&=\Scal{\widetilde \TR}{\Pi_E\gamma}=\Scal{ \TR}{\Pi_{E_0}\Pi_E\gamma}\\
&=
\Scal{ \TR}{\Pi_{E_0}\Pi_E\Pi_{E_0}\gamma} = \Scal{ \TR}{\Pi_{E_0}\gamma} =\Scal{\TR}{\gamma}.
\end{align*}

Conversely, by \eqref{aug31 eq:1}, $\widehat \TFF \res \theta=0$ and $\widehat \TFF \res d\theta=0$. Thus, by
\eqref{sept18 eq:1}, for all $\omega\in \mathcal D^{h}(\mathcal U)$,
$$
\Scal{\widetilde{\widehat{\TFF }}}{\omega}  = \Scal{\widehat \TFF }{\Pi_{E_0}\omega}=
 \Scal{ \TFF  }{\Pi_E\Pi_{E_0}\omega}= \Scal{\TFF }{\omega}.
$$

\end{proof}

\begin{corollary}\label{normrumin}
In dimensions $0\le k\le n$, the maps $\TR\mapsto \widetilde \TR$ and $\TFF \mapsto \widehat \TFF $ are mass isometries. It follows that normal and flat currents in both theories correspond.
\end{corollary}

\begin{proof}
This follows from the chains of inequalities
\begin{align*}
\mc M_{\mathbb{H}}(\TR)&=\mc M_{\mathbb{H}}(\widehat{\widetilde \TR})\le \mc M(\widetilde \TR)\le \mc M_{\mathbb{H}}(\TR),\\
\mc M(\TFF )&=\mc M(\widetilde{\widehat{\TFF }})\le \mc M_{\mathbb{H}}(\widehat \TFF )\le \mc M(\TFF ).
\end{align*}
Since boundaries also correspond under the maps $\TR\mapsto \widetilde \TR$ and $\TFF \mapsto \widehat \TFF $, so do normal currents and flat currents.
\end{proof}

\subsection{Correspondence between integral currents}

The proof that integral currents correspond under the maps $\TR\mapsto \widetilde \TR$ and $\TFF \mapsto \widehat \TFF $ passes via metric integral currents, in a cycle
$$
\text{Rumin}\implies\text{Federer-Fleming}\implies\text{Ambrosio-Kirchheim}\implies\text{Rumin}.
$$
Two steps are provided here, the main step being postponed until Section \ref{AK versus FF}.

\begin{lemma}\label{Ruminint=>FFint}
Every Rumin integral current $\TR$ defines a Federer-Fleming horizontal integral current $\widetilde \TR$.
\end{lemma}

\begin{proof}
By Definition \ref{defintRumin}, $\TR=\sum T_j$ with $T_j=(\phi_j)_{\#}|\![K_j]\!|$ a Lipschitz current of integration of Rumin forms, as in Definition \ref{defLipRumin}. A fortiori, $\phi_j$ is Lipschitz with respect to the Riemannian distance, so $\tilde T_j=(\phi_j)_{\#}|\![K_j]\!|$ makes sense as a Federer-Fleming rectifiable current. Lemma \ref{dtheta} shows that $\tilde T_j$ is horizontal. This fact, the decomposition of Proposition \ref{E0h} and the convergence of $\sum_j T_j$ as Rumin currents imply the convergence of the series $\widetilde \TR=\sum_j\tilde T_j$ as Federer-Fleming currents. So $\widetilde \TR$ is a Federer-Fleming rectifiable current. Since it is normal, it is a (horizontal) integral current. Note that it follows that $\partial \widetilde \TR$ is again an integral current.
\end{proof}

\begin{lemma}
Every integral metric current $\TAK$ defines a Rumin integral current $T_R$.
\end{lemma}

\begin{proof}
According to the parametric representation theorem (Theorem 4.5 in \cite{AK_acta}), $\TAK$ is a countable sum of metric currents of the form $f_\#(A)$ for Lipschitz maps $f:A\to\he{n}$, $A\subset \R^k$, with additivity of masses. Each summand in particular defines a Lipschitz current of integration on Rumin forms, as in Definition \ref{defLipRumin}. The convergence as metric currents being stronger than in the sense of distributions, this provides us with a Rumin integral current $T_R$ according to \ref{defintRumin}. 

\end{proof}

\begin{corollary}\label{integral}
Integral currents correspond under the maps $\TR\mapsto \widetilde \TR$ and $\TFF \mapsto \widehat \TFF $ between Rumin currents and horizontal Federer-Fleming currents.
\end{corollary}

\begin{proof}
The missing link is provided by Corollary \ref{intFF=>intmetric}.
\end{proof}

\subsection{High dimensions}

In this section, it will be convenient to view currents on an open set $\mathcal{U}$ as differential forms $\alpha$ with distributional coefficients, acting on test forms $\omega$ via
\begin{align*}
\int_{\mathcal{U}}\alpha\wedge\omega.
\end{align*}
Since there are two settings, Federer-Fleming (see Notation \ref{notFF}) and Rumin currents, a specific notation is introduced for Rumin currents. 

\begin{notation}
Let $\mathcal{U}\subset\he{n}$ be an open set. For a Rumin differential form $\beta$ with distributional coefficients on $\mathcal{U}$, let $\Ru(\beta)$ denote the \Rumin current defined by
\begin{align*}
\Scal{\Ru(\beta)}{\gamma}:=\int_{\mathcal{U}}\beta\wedge\gamma.
\end{align*}
\end{notation}

Lemma \ref{dFF} has a Rumin version, which follows from the integration by parts formula for Rumin forms, \cite{BFTT},
\begin{align*}
\int_{\mathcal{U}}(d_c\beta)\wedge\gamma=(-1)^{\text{degree}(\beta)+1}\int_{\mathcal{U}}\beta\wedge(d_c\gamma).
\end{align*}
\begin{lemma}\label{dRu}
Let $0\le k \le 2n+1$. Let $\beta$ be a smooth Rumin form of degree $2n+1-k$ on an open set $\mathcal{U}$. Then
\begin{align*}
\partial_{\mathbb{H}}\Ru(\beta)=(-1)^{k}\Ru(d_c\beta).
\end{align*}
\end{lemma}

Both maps $\FF$ and $\Ru$ are bijections, so one can write
\begin{align*}
\FF \,d=(-1)^k \partial\,\FF,&\quad d\,\FF^{-1} =(-1)^k \FF^{-1}\,\partial,\\ 
\Ru \,d_c=(-1)^k \partial_{\mathbb{H}}\,\Ru,&\quad \Ru^{-1}\,\partial_{\mathbb{H}}=(-1)^k d_c\,\Ru^{-1}.
\end{align*}

\begin{proposition}\label{identificationcurrentshigh}
Let $n+1\leq k \leq 2n+1$. Every \Rumin $k$-current $\TR\in \mathcal D_{\he{},k}(\mathcal U)$
can be identified with an oblique Federer-Fleming $k$-current $\widetilde \TR\in \mathcal D_{k}(\mathcal 
U)$ by setting
\begin{equation*}
\widetilde \TR:=\FF\,\Pi_{E}\,\Ru^{-1}(\TR).
\end{equation*}
Moreover, $\mc M(\widetilde \TR\res\theta)\le \mc M_{\mathbb H}(\TR)$ and $\partial\widetilde \TR=\widetilde{\partial_{\he{}}\TR}$.

Conversely, if $\TFF \in\mathcal D_{k}(\mathcal U)$ is an oblique Federer-Fleming current, then $\TFF $ induces a \Rumin $k$-current
$\widehat \TFF \in \mathcal D_{\he{},k}(\mathcal U)$ by 
$$
\widehat \TFF :=\Ru\,\Pi_{E_0}\,\FF^{-1}(\TFF ).
$$
Then $\mc M_{\mathbb H}(\widehat \TFF )\le  \mc M(\TFF \res \theta)$ and $\partial_{\he{}}\widehat \TFF =\widehat{\partial \TFF }$.

\end{proposition}

\begin{proof}
Let $h=2n+1-k\le n$. Let us assume first that $\beta:=\Ru^{-1}(\TR)$ is smooth. Then $\beta$ can be viewed as a horizontal $h$-form and $\Pi_E(\beta)-\beta$ is vertical. Hence, for every $\omega\in\mc D_{k-1}(\mc U)$,
\begin{align*}
\Scal{\widetilde \TR\res\theta}{\omega}
=\Scal{\Pi_E(\beta)}{\theta\wedge\omega}=\int_{\mathcal{U}}\Pi_E(\beta)\wedge\theta\wedge\omega=\int_{\mathcal{U}}\beta\wedge\theta\wedge\omega.
\end{align*}

Let $\Lambda$ denote the adjoint of the operator $L$ of multiplication with $d\theta$ on $\covh{\cdot}$ (Lefschetz operators). Recall that 
\begin{align*}
\mathcal{R}(L^{k-n+1})=\mathcal{R}(L^{n-h})=\mathcal{R}(\Lambda),
\end{align*}
whose orthogonal is $\ker(L)$. Thus, for every test form $\omega\in\mc D_k(\mc U)$, $\theta\wedge\omega$ splits orthogonally as
\begin{align*}
\theta\wedge\omega=\theta\wedge(d\theta)^{k-n+1}\wedge\omega_1+\theta\wedge\omega_2,
\end{align*}
where $\omega_2\in\ker(L)$ and
\begin{align*}
|\theta\wedge\omega_2|\le|\omega|.
\end{align*}
Note that since $L(\omega_2)=0$, $\theta\wedge\omega_2\in\mc D(\mc U,E_0^{k-1})$. According to Lemma \ref{check}, $\widetilde \TR$ is vertical, and therefore satisfies $\widetilde \TR \res (\theta\wedge (d\theta)^{k-n+1})=0$. It follows that
\begin{align*}
\Scal{\widetilde \TR\res\theta}{\omega}
&=\Scal{\widetilde \TR\res\theta}{\omega_2}
=\int_{\mathcal{U}}\beta\wedge\theta\wedge\omega_2\\
&=\Scal{\TR}{\theta\wedge\omega_2}\le \mc M_{\mathbb{H}}(\TR)\|\omega\|_{\infty}.
\end{align*}
Since every Rumin current $\TR$ is a weak limit of smooth Rumin currents $T_j$ with $M_{\mathbb{H}}(T_j)\le M_{\mathbb{H}}(\TR)$, this inegality persists for every Rumin current $\TR$, showing that
\begin{align*}
\mc M(\widetilde \TR\wedge\theta)\le \mc M_{\mathbb{H}}(\TR).
\end{align*}
The chain complex identity is straightforward,
\begin{align*}
\partial\widetilde \TR&=\partial\,\FF\,\Pi_E\,\Ru^{-1}(\TR)=(-1)^k \FF\,d\,\Pi_E\,\Ru^{-1}(\TR)\\
&=(-1)^k \FF\,\Pi_E\,d\,\Ru^{-1}(\TR)=\FF\,\Pi_E\,\Ru^{-1}(\partial_{\mathbb{H}}\TR)=\widetilde{\partial_{\mathbb{H}}\TR}.
\end{align*}

Conversely, let $\TFF =\FF(\alpha)$ be a smooth Federer-Fleming $k$-current. According to Proposition \ref{E0h}, since $h=2n+1-k\le n$, there exist differential forms $\beta'$ and $\gamma'$ such that
\begin{align*}
\alpha=\Pi_{E_0}\alpha+\alpha'\wedge\theta+\beta'\wedge d\theta.
\end{align*}
If $\gamma\in\mc D(\mc U,E_0^k)$, then $\gamma=\theta\wedge\gamma'$ where $\gamma'\in\mc D_{k-1}(\mc U)$ satisfies $d\theta\wedge\gamma'=0$ and $|\gamma'|=|\gamma|$. Hence
\begin{align*}
\Scal{\widehat \TFF }{\gamma}
&=\int_{\mathcal{U}}(\Pi_{E_0}\alpha)\wedge\gamma
=\int_{\mathcal{U}}\alpha\wedge\gamma
=\Scal{\TFF }{\gamma}\\
&=\Scal{\TFF \res\theta}{\gamma'}
\le M(\TFF \res\theta)\|\gamma\|_{L^{\infty}(\mathcal{U},E_0^k)}.
\end{align*}
By weak approximation of arbitrary Federer-Fleming currents of finite mass with smooth ones, this estimate persists for every current $\TFF $, hence
\begin{align*}
\mc M_{\mathbb{H}}(\widehat \TFF )\le \mc M(\TFF \res\theta).
\end{align*}

When $\TFF $ is an oblique current, the differential form $\FF^{-1}(\TFF )$ belongs to $E$, so $\FF^{-1}(\TFF )=\Pi_E \FF^{-1}(\TFF )$. Therefore
\begin{align*}
\partial_{\mathbb{H}}\widehat \TFF &=\partial_{\mathbb{H}}\,\Ru\,\Pi_{E_0}\,\FF^{-1}(\TFF )=(-1)^k \Ru\,d_c\,\Pi_{E_0}\,\FF^{-1}(\TFF )\\
&=(-1)^k \Ru\,\Pi_{E_0}\,d\,\Pi_{E}\,\Pi_{E_0}\,\Pi_E \, \FF^{-1}(\TFF )\\
&=(-1)^k \Ru\,\Pi_{E_0}\,d\,\Pi_E \, \FF^{-1}(\TFF )\\
&=(-1)^k \Ru\,\Pi_E\,d\,\FF^{-1}(T)=\Ru\,\Pi_{E_0}\,\FF^{-1}(\partial T)=\widehat{\partial \TFF}.
\end{align*}

\end{proof}

\begin{proposition}\label{involutivehigh} 
If $n+1\le k\le 2n+1$ and $\TR\in\mathcal D_{\he{},k}(\mathcal U)$ is a Rumin current, then
$$
\widehat {\widetilde \TR} =\TR.
$$
Conversely, if $\TFF \in \mathcal D_{k}(\mathcal U)$ is oblique, then
$$
\widetilde{\widehat{\TFF }} = \TFF .
$$

\end{proposition}

\begin{proof} 
Since $\Ru^{-1}(\TR)$ is a Rumin form, $\Pi_{E_0}\Ru^{-1}(\TR)=\Ru^{-1}(\TR)$. Thus
\begin{align*}
\widehat{\widetilde \TR}
&= \Ru\,\Pi_{E_0}\,\FF^{-1}\,\FF\,\Pi_{E}
\,\Ru^{-1}(\TR)
=\Ru\,\Pi_{E_0}\,\Pi_{E}
\,\Ru^{-1}(\TR)\\
&=\Ru\,\Pi_{E_0}\,\Pi_{E}
\,\Pi_{E_0}\,\Ru^{-1}(\TR)
=\Ru\,\Pi_{E_0}\,\Ru^{-1}(\TR)\\
&=\Ru\,\Ru^{-1}(\TR)=\TR.
\end{align*}

Conversely, if $\TFF $ is a vertical current, the differential form $\FF^{-1}(T)$ belongs to $E$, so $\Pi_E \FF^{-1}(T)=\FF^{-1}(T)$. Thus
\begin{align*}
\widetilde{\widehat{\TFF }}
&= \FF\,\Pi_{E}\,\Ru^{-1}\,\Ru\,\Pi_{E_0}
\,\FF^{-1}(\TFF )
= \FF\,\Pi_{E}\,\Pi_{E_0}
\,\FF^{-1}(\TFF )\\
&= \FF\,\Pi_{E}\,\Pi_{E_0}
\,\Pi_{E}\,\FF^{-1}(\TFF )
= \FF\,\Pi_{E}\,\FF^{-1}(\TFF )\\
&=\FF\,\FF^{-1}(\TFF )=\TFF .
\end{align*}
\end{proof}

Our next goal is to see what the procedure $\TR\mapsto \widehat{\TR}$ does on a current of integration on a smooth submanifold.

\begin{notation} \label{notBB}
Let $\BB$ denote the operator on Federer-Fleming currents that maps $T$ to the current
\begin{align*}
\BB(T):\omega\mapsto \Scal{T}{d_0^{-1}\omega}.
\end{align*}
\end{notation}

\begin{lemma} \label{PiEcurrents}
Let $T=\FF(\alpha)$ be a Federer-Fleming current. The oblique current $\bar T:=\FF(\Pi_E(\alpha))$ is given by
\begin{align*}
\bar T=T-\partial\BB(T)-\BB(\partial T).
\end{align*}
\end{lemma}

\begin{proof}
Let $k=\mathrm{dim}(T)$ and $h=2n+1-k=\mathrm{degree}(\alpha)$. By definition, $\Pi_E=Id-d_0^{-1}d-dd_0^{-1}$.
According to Lemma \ref{Hodge} (4), 
\begin{align*}
\FF(d_0^{-1}\alpha)=(-1)^{h}\BB(\FF(\alpha)=(-1)^{h}\BB(T),
\end{align*}
hence, with Lemma \ref{dFF},
\begin{align*}
\FF(d_0^{-1}d\alpha)&=(-1)^{h+1}\BB(\FF(d\alpha))=(-1)^{h+1+k}\BB(\partial \FF(\alpha))=\BB(\partial T),\\
\FF(dd_0^{-1}\alpha)&=(-1)^{k-1}\partial\FF(d_0^{-1}\alpha))=(-1)^{k-1+h}\partial\BB( \FF(\alpha))=\partial\BB(T).
\end{align*}
Hence
\begin{align*}
\bar T
&=\FF(\Pi_E \alpha)
=\FF(\alpha)-\FF(d_0^{-1}d\alpha)-\FF(dd_0^{-1}\alpha))\\
&=T-\partial\BB(T)-\BB(\partial T).
\end{align*}
\end{proof}

\begin{example} \label{exoblique}
Let $k=n+1,\ldots,2n+1$, let $h=2n+1-k$. Let $f:\mc U\to \R^h$ be a smooth map whose differential restricted to the contact hyperplane is onto. Let $\Phi:\R^h\to\R$ be a smooth function. Consider the averaged Rumin current of integration defined on $\mc D(\mc U,E_0^k)$, by
\begin{align*}
T_{R,f,\Phi}:\gamma \mapsto \int_{\R^h}\left(\int_{f^{-1}(y)}\gamma\right)\,\Phi(y)\,dy.
\end{align*}
Then 
  \begin{align*}
\widehat{T_{R,f,\Phi}}=T_{f,\Phi}-\partial \BB(T_{f,\Phi}),
\end{align*}
  where $T_{f,\Phi}$ is the averaged Federer-Fleming current of integration on level-sets of $f$ (see Example \ref{excolegendrian}), and the operator $\BB$ is defined in Notation \ref{notBB}.
\end{example}

Letting $\Phi$ approach a Dirac mass, one gets an expression for the oblique version of the Rumin current of integration along a smooth submanifold without boundary, transverse to the contact structure.

\subsection{The oblique mass}

The oblique mass makes sense for all Federer-Fleming currents in all dimensions. It is especially meaningful for oblique currents.

\begin{definition}\label{defvertmass}
The \emph{oblique mass} of a Federer-Fleming current $\TFF $ is 
$$
\mc M(\TFF \res \theta).
$$
\end{definition}

\begin{corollary}\label{normhigh}
In dimensions $k+1\le k\le 2n+1$, under the maps $\TR\mapsto \widetilde \TR$ and $\TFF \mapsto \widehat \TFF $ from Rumin to oblique Federer-Fleming currents, the Rumin mass corresponds exactly to the oblique mass.
\end{corollary}

\begin{proof}
This follows from the chains of inequalities
\begin{align*}
\mc M_{\mathbb{H}}(\TR)&=\mc M_{\mathbb{H}}(\widehat{\widetilde \TR})\le \mc M(\widetilde \TR\res\theta)\le \mc M(\TR),\\
\mc M(\TFF \res\theta)&=\mc M(\widetilde{\widehat{\TFF }}\res\theta)\le \mc M_{\mathbb{H}}(\widehat \TFF )\le \mc M(\TFF \res\theta).
\end{align*}

\end{proof}

In Heisenberg groups, the oblique mass arises when one rescales a Federer-Fleming current using anisotropic dilations.

\begin{lemma}\label{rescale}
\label{mass}
Let $\TFF$ be a $k$-dimensional Federer-Fleming current on $\he{n}$. Then
\begin{align*}
\mc M((\mathfrak s_\lambda)_{\#}\TFF)\le \lambda^{k+1}\mc M(\TFF\res\theta)+\lambda^{k}\mc M(\TFF).
\end{align*}
\end{lemma}

\begin{proof}
Let $\omega$ be a test $k$-form, expressed as
\begin{align*}
\omega=\theta\wedge\phi+\psi,
\end{align*}
with $\phi$ and $\psi$ horizontal, and $|\omega|^2=|\phi|^2+|\psi|^2$ pointwise. Then 
\begin{align*}
(\mathfrak s_\lambda)^*\omega=\lambda^{k+1}\theta\wedge\phi\circ \mathfrak s_\lambda+\lambda^{k}\psi\circ \mathfrak s_\lambda,
\end{align*}
where
\begin{align*}
\|\phi\circ \mathfrak s_\lambda\|_{\infty}=\|\phi\|_{\infty},\quad \|\psi\circ \mathfrak s_\lambda\|_{\infty}=\|\psi\|_{\infty}.
\end{align*}
So
\begin{align*}
|\Scal{(\mathfrak s_\lambda)_{\#}\TFF}{\omega }|
&=|\Scal{\TFF}{(\mathfrak s_\lambda)^*\omega }|\\
&\le |\Scal{\TFF\res\theta}{\phi\circ \mathfrak s_\lambda}|+|\Scal{\TFF}{\psi\circ \mathfrak s_\lambda }|\\
&\le \lambda^{k+1}\mc M(\TFF\res\theta)\|\phi\|_{\infty}+\lambda^{k}\mc M(\TFF)\|\psi\|_{\infty}\\
&\le (\lambda^{k+1}\mc M(\TFF\res\theta)+\lambda^{k}\mc M(\TFF))\|\omega\|_{\infty}.
\end{align*}
This shows that $\mc M((\mathfrak s_\lambda)_{\#}\TFF)\le \lambda^{k+1}\mc M(\TFF\res\theta)+\lambda^{k}\mc M(\TFF)$.
\end{proof}

To conclude this section, we give a geometric interpretation of the oblique mass for a current of integration on a $C^1$ submanifold. 

\begin{proposition}
1. In all dimensions, for $C^1$ submanifolds of contact manifolds, the oblique mass is a subRiemannian invariant: it depends only on the contact structure as a field of tangent hyperplanes and on the quadratic forms on hyperplanes. 

2. Viewed as a measure, the oblique mass restricted to subsets of a fixed $k$-dimensional $C^1$ submanifold has a continuous density with respect to the spherical Hausdorff measure $\mathcal{S}^{k+1}$. 

3. Unless it is horizontal, the current of integration $S$ on a $C^1$ submanifold satisfies
\begin{align*}
\mc M(S\res \theta)=\lim_{\lambda\to\infty}\lambda^{-k-1}\mc M((\mathfrak s_\lambda)_\# S).
\end{align*}

\end{proposition}

\begin{proof}
1. In a Riemannian manifold $V$, let $\theta$ be a smooth unit differential $1$-form. Let $Z$ denote the dual vectorfield. Let $S$ be a $C^1$ submanifold. At some point $p$ of $S$, assume that the tangent space $\tau=T_p S$ to $S$ is transverse to the horizontal space $\eta=\mathrm{Ker}(\theta_{p})$. Let $\beta$ be a $k-1$-vector (wedge of an orthonormal basis) dual to $\eta\cap\tau$ and $\nu_\tau$ a unit normal to $\eta\cap\tau$ in $\tau$. Let $\nu_\eta$ be a unit vector of $\mathrm{span}\{Z,\nu_\tau\}$ orthogonal to $Z$. Let $\alpha$ be the angle such that $\nu_\tau=\sin(\alpha)Z+\cos(\alpha)\nu_\eta$. Then, for any horizontal differential $k-1$-form $\phi$,
\begin{align*}
|(\theta\wedge\phi)(\nu_\tau\wedge\beta)|=|\sin(\alpha)\phi(\beta)|.
\end{align*}
Therefore, if $S$ still denotes the current of integration on $S$,
\begin{align*}
\mc M(S\res \theta)=\int_{S}|\sin(\alpha)|\,d\mathrm{area}.
\end{align*}
On the other hand, at $p$, the Riemannian $k$-area element can be written
\begin{align*}
d\mathrm{area}=|\theta\wedge\psi|,
\end{align*}
where $\psi$ is the horizontal $k-1$-covector such that
\begin{align*}
|(\theta\wedge\psi)(\nu_\tau\wedge\beta)|=1.
\end{align*}
Whence the expression
\begin{align*}
\mc M(S\res \theta)=\int_{S}|\theta\wedge\frac{\psi}{\psi(\beta)}|,
\end{align*}
which does not depend on the ambient Riemannian metric making $|\theta|=1$, only on the subRiemannian data.

2. When $V=\he{n}$, we see that the quantity $\mc M(S\res \theta)$, for $k$-dimensional $C^1$ submanifolds $S$, is invariant under Heisenberg translations and rotations, and homogeneous of degree $k+1$ under Heisenberg dilations. It follows that for every $p\in S$,
\begin{align*}
r^{-k-1}\mc M(S\res 1_{B(p,r)} \theta)
\end{align*}
converges as $r\to 0$ to an invariant $c$ of horizontal $k$-planes which is invariant under translations and rotations, and in turn, that the measure on Borel subsets $A\subset S$ given by
\begin{align*}
A\mapsto \mc M(A\res \theta)
\end{align*}
has density $c$ with respect to the subRiemannian spherical measure $\mathcal{S}^{k+1}$. Note that $c$ is constant when $n=1$.

3. Large Heisenberg dilations tend to increase up to $\frac{\pi}{2}$ the angle between tangent planes and the horizontal plane. Hence, when $S$ is nowhere horizontal, 
\begin{align*}
\lim_{\lambda\to\infty}\frac{\mc M((\mathfrak s_\lambda)_{\#}S\res\theta)}{\mc M((\mathfrak s_\lambda)_{\#}S)}=1.
\end{align*}
By subRiemannian invariance,
\begin{align*}
\mc M((\mathfrak s_\lambda)_{\#}S\res\theta)=\lambda^{k+1}\mc M(S\res\theta),
\end{align*}
whence the announced asymptotics. In general, $S=S+S''$, where $S'$ is nowhere horizontal and $S''$ is horizontal. Assuming that $S'\not=0$,
\begin{align*}
\mc M((\mathfrak s_\lambda)_{\#}S''\res\theta)&=0,\quad \mc M((\mathfrak s_\lambda)_{\#}S'')=\lambda^{k}\mc M(S''),\\
&\lim_{\lambda\to\infty}\frac{\mc M((\mathfrak s_\lambda)_{\#}S'\res\theta)}{\mc M((\mathfrak s_\lambda)_{\#}S')}=1,\quad 
\end{align*}
which implies again that
\begin{align*}
\mc M((\mathfrak s_\lambda)_{\#}S)\sim \lambda^{k+1}\mc M(S\res\theta).
\end{align*}
Quantitatively, $\mc M((\mathfrak s_\lambda)_\# S)=\mc M_\lambda(S)$, where $\mc M_\lambda$ is the mass associated with the Riemannian metric $g_\lambda:=\lambda^{2} g_\eta+\lambda^4\theta^2$. With respect to this metric, the unit horizontal $1$-form is $\theta_\lambda:=\lambda^{2} \theta$, the dual vectorfield is $Z_\lambda:=\lambda^{-2}Z$. The unit $k-1$-vector associated with $\tau\cap\eta$ is $\beta_\lambda:=\lambda^{-(k-1)}\beta$. Since the $2$-plane $\mathrm{span}\{Z,\nu_\eta\}$ is $g_\lambda$-orthogonal to $\eta\cap\tau$, it contains the unit normal $\nu_{\tau,\lambda}$ to $\eta\cap\tau$ in $\tau$, one can take $\nu_{\eta,\lambda}:=\lambda^{-1}\nu_{\eta}$ and $\nu_{\tau,\lambda}$ is proportional to $\nu_{\tau}$, 
\begin{align*}
\nu_{\tau,\lambda}:=((\lambda^2 \sin\alpha)^2+(\lambda\cos\alpha)^2)^{-1/2}\nu_\tau.
\end{align*}
Therefore
\begin{align*}
|(\theta\wedge\psi)(\nu_{\tau,t}\wedge\beta_\lambda)|=((\lambda^2 \sin\alpha)^2+(\lambda\cos\alpha)^2)^{-1/2}\lambda^{-(k-1)},
\end{align*}
hence the area element induced by $g_\lambda$ on $S$ is $|\theta\wedge\psi_\lambda|$ with
\begin{align*}
\psi_\lambda:=((\lambda^2 \sin\alpha)^2+(\lambda\cos\alpha)^2)^{1/2}\lambda^{k-1}\psi.
\end{align*}
This leads to 
\begin{align*}
\mc M_\lambda(S)&=\int_{f(\Delta^k)}\lambda^{k-1}\sqrt{(\lambda^2 \sin\alpha)^2+(\lambda\cos\alpha)^2}\,d\mathrm{area}\\
&=\lambda^{k}\int_{S}\sqrt{(\lambda \sin\alpha)^2+(\cos\alpha)^2}\,d\mathrm{area}.
\end{align*}
This allows to recover the estimate of Lemma \ref{mass} in the special case of immersed $C^1$ submanifolds, and shows in addition that, as $\lambda\to\infty$, the first term $\lambda^{k+1}\mc M(S\res\theta)$ is an asymptotic.
\end{proof}

\section{The random deformation theorem}\label{random}

\begin{definition}  \label{defC1chain}
Let $k=0,\ldots,2n+1$. Let $\Delta^k\subset \R^k$ denote the oriented regular Euclidean $k$-simplex. A \emph{$C^1$ simplex} (resp. \emph{Lipschitz simplex}) is a map $f:\Delta^k \to \he{n}$ of class $C^1$ (resp. Lipschitz with respect to the left-invariant Riemannian metric on $\he{n}$). 

A  $C^1$ (resp. \emph{Lipschitz}) \emph{$k$-chain} in $\he n$ is a finite sum of the form
$$
\sum_i a_i(f_i)_\#(\mathcal{L}^k \res \Delta^k)
$$
where $a_i \in\mathbb Z$ and $f_i: \Delta^k \to \he n$ are $C^1$ (resp. Lipschitz) simplices, viewed as a Federer-Fleming integral current.

\end{definition}

The following fact is general for Riemannian manifold admitting discrete cocompact groups of isometries, the details are provided in \cite{contactfilling}.

\begin{lemma}
There exists a constant $L=L(n)$ such that, for every  $\eps>0$, the Heisenberg group $\he n$ admits an $(\eps,L)$-triangulation, i.e. a triangulation whose simplices are $L$-biLipschitz to regular simplices of sidelength $\eps$.
\end{lemma}

Here is a modification of Federer \& Fleming's Deformation Theorem. In Euclidean space, it would be a mere restatement of an intermediate step of the classical construction, a step where algebraic properties still hold, a fact that will be crucial in Section \ref{AK versus FF}. We merely state a consequence customized for Riemannian Heisenberg groups.

\begin{proposition}[Random Deformation Theorem, \cite{contactfilling}]\label{deformation}
Let $\kappa$ be an $(\eps,L(n))$-triangulation of $\he n$. Then the group $I$ of Federer-Fleming integral currents in $\he n$ and the group $I_\kappa$ of chains of the triangulation $\kappa$ are randomly chain homotopic, with uniform mass bounds in expectation. 

Namely, there exists a constant $C=C(n)$ and random self-maps $P$ and $Q:I\to I$ such that
\begin{enumerate}
  \item $\partial P=P\partial$ ;
  \item $P+Q\partial+\partial Q$ is the identity map ;
  \item $P$ and $Q$ are additive: for every integral currents $S,S'$, $P(S+S')=P(S)+P(S')$, $Q(S+S')=Q(S)+Q(S')$.
  \item $P$ is a projector onto the subgroup $I_\kappa\subset I$: $P(I)\subset I_\kappa$ and $P\circ P=P$ ;
  \item For every integral current $S\in I$,
\begin{align*}
\E(\mc M(P(S)))&\le C\,(\mc M(S)+\epsilon \mc M(\partial S)),\\
\E(\mc M(Q(S)))&\le C\epsilon\,\mc M(S).
\end{align*}
\end{enumerate}
Furthermore, $P(S), Q(S)$ have support in a $C\eps$-neighborhood of the support of $S$. 

Finally, if $S$ is a $C^1$ (resp. Lipschitz) chain, so are $P(S)$ and $Q(S)$.
\end{proposition}

\section{Ambrosio-Kirchheim metric currents versus horizontal
Federer-Fleming currents}\label{AK versus FF}

At the level of normal currents, the connection between metric and horizontal Federer-Fleming currents parallels the corresponding discussion in Euclidean spaces (\cite{AK_acta}, Section 5) and holds more generally for Carnot groups, as shown in \cite{Williams}, Theorem 1.6. 
Nevertheless, for completeness'sake, the details in the case of Heisenberg groups are provided here in Sections \ref{AK->FF} and \ref{FF->AK}.

\subsection{From metric to Federer-Fleming}
\label{AK->FF}

By definition (\cite{AK_acta}), a metric $k$-current is a multilinear functional $(f,\pi_1,\ldots\pi_k)\mapsto \TAK(f,\pi_1,\ldots\pi_k)$ on $k+1$-tuples of Lipschitz functions satisfying a locality axiom and a continuity axiom.
Therefore, a Federer-Fleming current $\widetilde \TAK$ is easily associated with $\TAK$ as follows. Fix smooth coordinates on $\he{n}$, for instance exponential coordinates $x^i$ where the last one $x^{2n+1}$ is the vertical coordinate, the others being horizontal and orthonormal. Then every smooth differential form $\omega$ has a unique expression
$$
\omega=\sum_I \omega_I dx^I
$$
as a sum of simple covectors.

Let $\TAK$ be a metric $k$-current on $\he{n}$. For $\omega$ a test $k$-form, let
$$
\Scal{\widetilde \TAK}{\omega }:=\sum_I \TAK(\omega_I,x_{i_1},\ldots,x_{i_k}),
$$
where the sum is extended over multiindices $I=(i_1,\ldots,i_k)$.

\textbf{Claim}. This defines a Federer-Fleming current $\widetilde \TAK$.

This follows from the continuity axiom in the definition of metric currents.

\textbf{Claim}. $\partial \widetilde \TAK=\widetilde{\partial \TAK}$.

By definition, for every test form $\omega$,
\begin{align*}
(\partial \TAK)(\omega)=\sum_{i,I}\TAK(\frac{\partial \omega_I}{\partial x_i},x_i,x_{i_i},\ldots,x_{i_k}),
\end{align*}
so
\begin{align*}
\widetilde{\partial \TAK}(\omega)&=\sum_{i,I}\Scal{\widetilde \TAK}{\frac{\partial \omega_I}{\partial x_i},x_i,x_{i_i},\ldots,x_{i_k}}\\
&=\Scal{\widetilde \TAK}{\sum_{i,I}\frac{\partial \omega_I}{\partial x_i},x_i,x_{i_i},\ldots,x_{i_k}}\\
&=\Scal{\widetilde \TAK}{d\omega}=\Scal{\partial\widetilde \TAK}{\omega}.
\end{align*}

\textbf{Claim}. If $\omega=g\,d\tau_1\wedge\cdots\wedge d\tau_\ell$, with $g$ bounded and $\tau_i$ Lipschitz functions on the support of $\TAK$, then $\widetilde{\TAK\res \omega}=\widetilde \TAK\res\omega$. If $\TAK$ has finite mass, so does $\TAK\res\omega$.

This is by design, the metric restriction mimics the Federer-Fleming restriction. The mass bound appears in Equation (2.5) in \cite{AK_acta})

\textbf{Claim}. $\mc M(\widetilde \TAK)\le C\,\mc M(\TAK)$, where $C$ depends on $n$ and on the support of $\TAK$.

Indeed, let $K$ denote the support of $\TAK$. There exists a constant $C$ such that, on $K$,
$$
\prod_{j=1}^k |dx_{i_j}|\le C(K)\,|\bigwedge_{j=1}^k dx_{i_j}|.
$$
Since $dist_c\ge dist$, subRiemannian Lipschitz constants $Lip_c$ are controlled by Riemannian ones, so, for simple differential forms $\omega=f\,dx_{i_1}\wedge\cdots\wedge dx_{i_k}$,
\begin{align*}
|\Scal{\widetilde \TAK}{\omega }| &=|\TAK(f,x_{i_1},\ldots,x_{i_k})|\\
&\le \mc M(\TAK)\|f\|_{\infty}\prod_{j=1}^k Lip_c(x_{i_j})\\ 
&\le \mc M(\TAK)\|f\|_{\infty}\prod_{j=1}^k |dx_{i_j}|\\ 
&\le C\,\mc M(\TAK)\|f\|_{\infty}|\bigwedge_{j=1}^k dx_{i_j}|\\
&=C\,\mc M(\TAK)\|\omega\|_\infty.
\end{align*}
Adding up boundedly many terms yields a similar estimate for general differential forms.

\begin{lemma}
If the metric current $\TAK$ has finite mass, then the corresponding Federer-Fleming current satisfies $\widetilde \TAK\res\theta=0$.
\end{lemma}

\begin{proof}
Let $\phi$ be a horizontal test $k-1$-form. Then, under the dilations $(\mathfrak s_\lambda)$, $\phi$ grows like $\lambda^{k-1}$, $\theta$ grows like $\lambda^2$, whereas, since the dilatations are metric homotheties, the mass of $\TAK$ grows like $\lambda^k$. This implies that $\Scal{\TAK}{\theta\wedge\phi }=0$.

Here are the details. One can write
\begin{align*}
\phi=\sum_{I\subset\{1,2n\},\,|I|=k-1}\phi_I dx^I,
\end{align*}
where the last coordinate $x_{2n+1}$ is absent. Then
\begin{align*}
\lambda^{-k+1}\mathfrak s_\lambda^{*}\phi=\sum_{I}(\phi_I \circ \mathfrak s_\lambda) \,dx^I.
\end{align*}
We observe that
\begin{align*}
\|\phi_I \circ \mathfrak s_\lambda\|_{\infty}=\|\phi_I\|_{\infty}\le \|\phi\|_{\infty} \quad\text{and}\quad Lip_c(x_I)\le 1.
\end{align*}
It follows that for every metric $k-1$-current $S$,
\begin{align*}
S(\lambda^{-k+1}\mathfrak s_\lambda^{*}\phi)\le C\,\mc M(S)(\sum_I\|\phi_I\|_{\infty}\prod_j Lip_c(x_{i_j})).
\end{align*}
On the other hand,
\begin{align*}
\lambda^{-2}(\mathfrak s_\lambda)^*\theta=\theta,
\end{align*}
hence
\begin{align*}
\lambda^{-2}(\mathfrak s_\lambda)_{\#}(\TAK\res\theta)=((\mathfrak s_\lambda)_{\#}\TAK)\res\theta,
\end{align*}
Since the Heisenberg dilatations $\mathfrak s_\lambda$ are homothetic, for every $k$-current $U$,
$$
\mc M((\mathfrak s_\lambda)_{\#}U)=\lambda^{k}\mc M(U).
$$
It follows that
$$
\lambda^{2-k}\mc M(((\mathfrak s_\lambda)_{\#}\TAK)\res \theta)=\mc M(\TAK\res\theta)<\infty.
$$

Hence
\begin{align*}
\lambda\,\Scal{\widetilde \TAK}{\theta\wedge\phi }
&=\lambda\,\Scal{(\mathfrak s_\lambda)_{\#}\widetilde \TAK}{(\mathfrak s_{1/\lambda})^*\theta\wedge(\mathfrak s_{1/\lambda})^*\phi }\\
&=\lambda^{2-k}\,(((\mathfrak s_\lambda)_{\#}\TAK)\res\theta)((\frac{1}{\lambda})^{-k+1}(\mathfrak s_{1/\lambda})^*(\sum\phi_I \circ  dx^{I}))\\
&\le \lambda^{2-k}\mc M(((\mathfrak s_\lambda)_{\#}\TAK)\res\theta)(\sum_I\|\phi_I\|_{\infty}\prod_j Lip_c(x_{i_j}))\\
&= \mc M(\TAK\res\theta)(\sum_I\|\omega_I\|_{\infty}\prod_j Lip_c(x_{i_j}))
\end{align*}
stays bounded as $\lambda$ tends to $\infty$. So $\Scal{\widetilde \TAK}{\theta\wedge\phi }=0$, and $\widetilde \TAK\res\theta=0$.

\end{proof}

\textbf{Claim}. If $\TAK$ is an integral metric current, then $\widetilde \TAK$ is an integral Federer-Fleming current.

Indeed, since $\widetilde{\partial \TAK}=\partial \widetilde \TAK$, $\TAK$ normal implies $\widetilde \TAK$ normal. According to the parametric representation theorem (Theorem 4.5 in \cite{AK_acta}), $\TAK$ is a countable sum of currents of the form $f_\#(A)$ for Lipschitz maps $f:A\to\he{n}$, $A\subset \R^k$, with additivity of masses. So is $\widetilde \TAK$, which is therefore integral.

\subsection{From Federer-Fleming to metric}
\label{FF->AK}

Conversely, we show that a normal horizontal Federer-Fleming current $\TFF $ defines a normal metric current $\widehat \TFF $. The point is to extend from smooth to Lipschitz test forms.

By Riesz' representation theorem we have
\begin{proposition}
If $\TFF \in \mathcal D_{k}(\mathcal U)$ has locally finite mass, 
there exists a Radon measure $\|\TFF \|$ on $\mc U$ and a function $\vec{T}_{FF}:\mc U\to\covH{k}$,
$|\vec{T}_{FF}|= 1$ $\|\TFF \|$-a.e. such that
\begin{itemize}
\item[i)] $\Scal{T}{\phi} = \int\scal{\vec{T}_{FF}}{\phi}\, d\|\TFF \|$ for all test form $\phi$;
\item[ii] $\|\TFF \|(\mc U) = \mc M(\TFF)$.
\end{itemize}
\end{proposition}

\begin{lemma}
\label{alphaH}
Let us denote by $\alpha^H$ the horizontal component of a covector $\alpha$.

Let $\alpha_1,\ldots,\alpha_k$ be smooth differential $1$-forms on $\he{n}$, let $f$ be a smooth function. Let $\TFF $ be a horizontal Federer-Fleming $k$-current. Then
$$
|\Scal{\TFF }{f\,\alpha_1\wedge\cdots\wedge\alpha_k }|\le C(n)\,\prod_{i=1}^k \|\alpha_i^H\|_\infty\,\int |f|\,d\|\TFF \|.
$$
\end{lemma}

\begin{proof}
By definition, there exists smooth functions $\lambda_i$ such that $\alpha_i=\alpha_i^H+\lambda_i\theta$. So there exists a $k-1$-form $\beta$ such that
$$
\alpha_1\wedge\cdots\wedge\alpha_k=\alpha^H_1\wedge\cdots\wedge\alpha^H_k+\theta\wedge\beta.
$$
Since $\TFF \res\theta=0$,
\begin{align*}
|\Scal{\TFF }{f\,\alpha_1\wedge\cdots\wedge\alpha_k }|&=|\Scal{\TFF }{f\,\alpha^H_1\wedge\cdots\wedge\alpha^H_k }|\\
&\le \|\alpha^H_1\wedge\cdots\wedge\alpha^H_k\|_\infty\,\int |f|\,d\|\TFF \|\\
&\le C\,\|\alpha^H_1\|_\infty\cdots\|\alpha^H_k\|_\infty\,\int |f|\,d\|\TFF \|,
\end{align*}
for some constant depending only on dimension.
\end{proof}

\begin{proposition}
\label{lipc}
Let $\pi_1,\ldots,\pi_k$ and $f$ be smooth functions on $\he{n}$. Let $\TFF $ be a horizontal Federer-Fleming $k$-current. Then
$$
|\Scal{\TFF }{f\,d\pi_1\wedge\cdots\wedge d\pi_k }|\le C(n)\,\prod_{i=1}^k Lip_c(\pi_i)\,\int |f|\,d\|\TFF \|.
$$
\end{proposition}

\begin{proof}
By definition of the subRiemannian metric on $\he{n}$, for a smooth function $\pi$, $Lip_c(\pi)=\|(d\pi)^H\|_\infty$. The estimate then follows from Lemma \ref{alphaH}.
\end{proof}

From now on, we follow closely arguments from \cite{AK_acta}, sections 5 and 11.

\begin{lemma}
\label{unifcont}
Let $f,\pi_1,\ldots,\pi_k$ and $f',\pi'_1,\ldots,\pi'_k$ be smooth functions on $\he{n}$ such that
$$
\forall i,\quad Lip_c(\pi_i)\le 1,\quad \text{and}\quad Lip_c(\pi'_i)\le 1.
$$ 
Let $\TFF $ be a horizontal Federer-Fleming $k$-current. Then
\begin{align*}
&|\Scal{\TFF }{f\,d\pi_1\wedge\cdots\wedge d\pi_k }-\Scal{\TFF }{f'\,d\pi'_1\wedge\cdots\wedge d\pi'_k }|\\&\le C(n)\,(\int|f-f'|\,d\|\TFF \|\\
& +\sum_{i=1}^k \int |f||\pi_i-\pi'_i|\,d\|\partial \TFF \|+Lip_c(f)\int |\pi_i-\pi'_i|\,d\|\TFF \|).
\end{align*}
\end{lemma}

\begin{proof}
Let us abbreviate $d\pi_2\wedge\cdots\wedge d\pi_k:=d\pi_0$ and integrate by parts
\begin{align*}
&|\Scal{\TFF }{f\,d\pi_1\wedge d\pi_0 }-\Scal{\TFF }{f\,d\pi'_1\wedge d\pi_0 }|\\
&=|\Scal{\TFF }{d(f(\pi_1-\pi'_1)\wedge d\pi_0) }-\Scal{\TFF }{(\pi'_1-\pi_1)\,df\wedge d\pi_0 }|.
\end{align*}
Applying Proposition \ref{lipc} allows to estimate the first term, 
\begin{align*}
|\Scal{\TFF }{d(f(\pi_1-\pi'_1)\wedge d\pi_0) }|
&=|\Scal{\partial \TFF }{f(\pi_1-\pi'_1)\wedge d\pi_0)}|\\
&\le C\,\prod_{i=2}^k Lip_c(\pi_i)\int|f||\pi_1-\pi'_1|\,d\|\partial \TFF \|\\
&\le C\,\int|f||\pi_1-\pi'_1|\,d\|\partial \TFF \|,
\end{align*}
since $Lip_c(\pi_i)\le 1$. The second term is estimated in a similar manner,
\begin{align*}
|\Scal{\TFF }{(\pi'_1-\pi_1)\,df\wedge d\pi_0 }|
&\le C\,Lip_c(f)\prod_{i=2}^k Lip_c(\pi_i)\,\int|\pi_1-\pi'_1|\,d\|\TFF \|\\
&\le C\,Lip_c(f)\int|\pi_1-\pi'_1|\,d\|\TFF \|.
\end{align*}
Adding the two terms yields
\begin{align*}
&|\Scal{\TFF }{f\,d\pi_1\wedge d\pi_0 }-\Scal{\TFF }{f\,d\pi'_1\wedge d\pi_0 }|\\
&\le C\,(\int|f||\pi_1-\pi'_1|\,d\|\partial \TFF \|+Lip_c(f)\int|\pi_1-\pi'_1|\,d\|\TFF \|).
\end{align*}
Repeating the argument with pairs $(\pi_i,\pi'_i)$, $i=2,\dots,k$, and adding up, leads to the sum on the right hand side. Finally, 
\begin{align*}
|\Scal{\TFF }{f\bigwedge_{i=1}^k d\pi'_i }&-\Scal{\TFF }{f'\bigwedge_{i=1}^k d\pi'_i }\le \int|f-f'|\,d\|\TFF \|.
\end{align*}
Adding this extra term to the sum yields the announced estimate for the difference
\begin{align*}
|\Scal{\TFF }{f\bigwedge_{i=1}^k d\pi_i }-\Scal{\TFF }{f'\bigwedge_{i=1}^k d\pi'_i }|.
\end{align*}

\end{proof}

By convolution, every Lipschitz function $f$ on $\he{n}$ is the limit of a sequence $f^\eps$ of smooth functions with uniformly bounded Lipschitz constants.
Lemma \ref{unifcont} shows that when smooth functions $f^{\eps},\pi^{\eps}_1,\ldots,\pi^{\eps}_k$ converge uniformly to Lipschitz functions $f,\pi_1,\ldots,\pi_k$ while keeping bounded Lipschitz constants, the numbers $\Scal{\TFF }{f\bigwedge_{i=1}^k d\pi_i }$ converge. Let us denote the limit by
\begin{align*}
\widehat \TFF (f,\pi_1,\ldots,\pi_k):=\lim_{\eps\to 0}\Scal{\TFF }{f\bigwedge_{i=1}^k d\pi_i }.
\end{align*}
This defines a metric functional, multilinear in $f,\pi_1,\ldots,\pi_k$. It is local in the sense that its support is contained in the support of $\TFF $. Since the convergence is uniform on $(Lip_c^1)^{k+1}$, the metric functional $\widehat \TFF $ satisfies the requested continuity axiom, hence it is a metric current. 

\textbf{Claim}. Mass estimate.
\begin{align*}
\|\widehat \TFF \|\le C(n)\,\|\TFF \|.
\end{align*}
In particular, 
\begin{align*}
\mc M(\widehat \TFF )\le C(n)\,\mc M(\TFF ).
\end{align*}
Indeed, for smooth data $f,\pi_1,\ldots,\pi_k$,
\begin{align*}
|\widehat \TFF (f,\pi_1,\ldots,\pi_k)|\le C(n)\,\prod_{i=1}^{k}Lip_c(\pi_i)\,\int|f|\,d\|\TFF \|.
\end{align*}
With Lemma \ref{unifcont}, this inequality passes to the limit and applies to arbitrary Lipschitz data $f,\pi_1,\ldots,\pi_k$. 

\textbf{Claim}. $\partial \widehat \TFF =\widehat{\partial \TFF }$.

Again, for smooth data $f,\pi_1,\ldots,\pi_k$,
\begin{align*}
(\partial\widehat \TFF )(f,\pi_1,\ldots,\pi_k)&=\widehat \TFF (1,f,\pi_1,\ldots,\pi_k)=\Scal{\TFF }{df\wedge \bigwedge_{i=1}^k d\pi_i}\\
& =\Scal{\TFF }{d(f\wedge \bigwedge_{i=1}^k d\pi_i)}=\Scal{\partial \TFF }{f\wedge \bigwedge_{i=1}^k d\pi_i }\\
&=\widehat{\partial \TFF }(f,\pi_1,\ldots,\pi_k).
\end{align*}
This extends to all Lipschitz data by the continuity property of metric currents, showing that $\partial \widehat \TFF =\widehat{\partial \TFF }$.

We conclude that if $\TFF $ is a normal horizontal Federer-Fleming current, then $\widehat \TFF $ is a normal metric current.

\subsection{Structure of Federer-Fleming horizontal integral currents}

Our aim is to show that if $\TFF $ is a horizontal integral current, then $\widehat \TFF $ is an integral metric current. There is a difficulty, stemming from the fact that the corresponding statement probably fails for rectifiable currents, as the following example indicates. It is a good candidate for a horizontal, Euclidean-rectifiable set, which would not be metric-rectifiable.

\begin{example}
\label{Cantor}
Start with a Cantor set $A$ of positive measure in $[0,1]$. Consider the continuous function $u(x)=\sqrt{d(x,A)}$ on $\R$. Let 
$$
f:\R\to \he1,\quad f(x)=(x,0,\int_{0}^{x}u(y)\,dy).
$$
Then $f$ is $C^1$ on $\R$, $\theta(f'(x))=u(x)$ for all $x\in \R$, so $f'$ is horizontal at each point of $A$. For a suitable choice of $A$, $f$ is not P-differentiable at some density point of $A$. 
\end{example}
Assume that $A$ is obtained as follows: let $(r_i)$ be a sequence of positive real numbers such that $\sum_{i=0}^{\infty}2^i r_i < 1$. Let $I_{\emptyset}=[0,1]$. Remove a symmetric open interval of length $r_0$ around $\frac{1}{2}$, getting two intervals $I_{0}$ and $I_{1}$, then symmetric open intervals of length $r_1$ at the centers of both intervals, and so on. Then $0$ is a density point of $A$. Let $I_{00\ldots0}=(a,b)$ denote the leftmost interval at the $j$th stage. Its length is $r_j$ and its position is roughly $2^{-j-1}$. The vertical projection of $f(b)$ satisfies
\begin{align*}
\int_{0}^{b}u(y)\,dy\ge \int_{a}^{b}u(y)\,dy=2\int_{0}^{r_j/2}\sqrt{y}\,dy=\frac{4}{3}(\frac{r_j}{2})^{3/2}.
\end{align*}
One can choose the sequence $(r_j)$ such that $r_j^{3/2}$ is not $o(2^{-2j})$. Then $f$ is not P-differentiable at $0$.

%

\subsection{Horizontal $C^1$ chains}

Our strategy is to show that horizontal $C^1$ chains are metric integral currents, and to approximate general horizontal Federer-Fleming currents with such $C^1$ chains.

\begin{notation}
Let $\mathcal{V}$ be an open subset of $\R^k$, let $A\subset \mathcal{V}$ be a Borel subset and $f:\mathcal{U}\to\he{n}$ a map of class $C^1$ in the usual sense. The Federer-Fleming current of integration on $f(A)$ is denoted by
$$
T_{f(A)}:=f_\#(\mathcal{L}^k\res A).
$$
\end{notation}

\begin{definition}
Let $k=0,\ldots,2n+1$. Let $\Delta^k\subset \R^k$ denote the oriented regular Euclidean $k$-simplex. In a Riemannian manifold $V$, a $C^1$ simplex is a $C^1$ map $f:\Delta^k \to V$. Given $\eps>0$, the simplex $f$ is $C^1_\eps$ if $\forall  x,y\in\Delta^k$,
\begin{align*}
 (1-\eps)d(x,y)\le d(f(x),f(y))\le (1+\eps)d(x,y)
\end{align*}

A $C^1$ (resp. $C^1_\eps$) chain is a finite $\Z$-linear combination of $C^1$ (resp. $C^1_\eps$) simplices, viewed as a Federer-Fleming integral current. 
\end{definition}

\begin{lemma}
\label{C1=>metric}
Horizontal $C^1$ chains in $\he{n}$ are metric integral currents.
\end{lemma}

\begin{proof}
Case of $C^1$ simplices. Let us handle a slightly more general case. Let $f:\Delta^k \to \he{n}$ be a $C^1$ simplex and $\mathcal{V}\subset\Delta^k$ be an open subset such that the Federer-Fleming current $T_{f(\mathcal{V})}$ is horizontal. Let us show that $T_{f(\mathcal{V})}$ is a metric rectifiable current. By assumption, for every test $k-1$-form $\phi$, 
$$
\int_{\mathcal{V}}f^*(\theta\wedge\phi)= \Scal{T_{f(\mathcal{V})}\res \theta}{\phi }=0.
$$ 
This identity extends to all continuous compactly supported $k-1$-form $\phi$.

Let $\mathcal{W}\subset \mathcal{V}$ be the open subset of points where $f$ is an immersion. Then $f^*(\theta\wedge\phi)=0$ on $\mathcal{V}\setminus \mathcal{W}$, so
\begin{align*}
\int_{\mathcal{W}} f^*(\theta\wedge\phi)=\int_{\mathcal{V}}f^*(\theta\wedge\phi)=0.
\end{align*}
Given a point $x_0\in \mathcal{W}$, let $B$ be an open ball centered at $x_0$, whose closure is contained in $\mathcal{W}$, and which is small enough so that the restriction of $f$ to $B$ is an embedding onto a $C^1$-submanifold $S$ of $\he{n}$. Let $\psi$ be a continuous $k-1$ form with support in $B$. Since $f$ is a $C^1$ diffeomorphism of $B$ onto $S$, the continuous compactly supported $k-1$-form $(f^{-1})^{*}\psi$ on $S$ has a continuous and compactly supported extension $\phi$ to $\he{n}$, hence
\begin{align*}
\int_{B\cap A} f^*(\theta)\wedge\psi=\int_{B\cap A} f^*(\theta\wedge\phi)=0.
\end{align*}
This shows that $f^{*}\theta=0$ almost everywhere on $B$, hence on all of $B$ by density. Therefore $f$ maps every line segment in $B$ of length $\ell$ to a $C^1$ horizontal curve with with length $\le \ell L$, where $L$ is the $C^1$ norm of $f$. By definition of the subRiemannian distance, this shows that $f$ is $L$-Lipschitz on $B$, hence, for every measurable subset $A\subset B$, the corresponding current $T_{f(A)}$ is a metric rectifiable current. It follows that $T_{f(\mathcal{V})}=T_{f(\mathcal{W})}$ is a metric rectifiable current as well. 

General case. Let $S=\sum\lambda_i T_{f_i(\Delta^k)}$ be a $C^1$ chain which is a horizontal Federer-Fleming integral current. This is nearly a mass decomposition, up to the following phenomenon: it may happen that certain simplices $f_i$ and $f_j$ overlap with $|\lambda_i+\lambda_j|<|\lambda_i|+|\lambda_j|$. One can extract a mass decomposition
\begin{align*}
S=\sum\lambda_i T_{f_i(\mathcal{V}_i)},\quad \mc M(S)=\sum_i|\lambda_i|\mc M(T_{f_i(\mathcal{V}_i)}),
\end{align*}
for some open subsets $\mathcal{V}_i\subset \Delta^k$. According to the following Lemma \ref{add}, each $T_{f_i(\mathcal{V}_i)}$ is horizontal, hence a metric rectifiable current. The linear combination $S$ is thus a metric rectifiable current as well. By induction on dimension, $\partial S$ is metric rectifiable, so $S$ is a metric integral current.
\end{proof}

\begin{lemma}\label{add}
Let $\TFF $ be a Federer-Fleming rectifiable current, admitting a mass decomposition of the form
\begin{align*}
\TFF =\sum_{j}T_{j}\quad \text{ with}\quad \mc M(\TFF )=\sum_{j}\mc M(T_{j}).
\end{align*}
If $\TFF  \res \theta=0$, then $T_j \res \theta=0$ for all $j$.
\end{lemma}

\begin{proof}
Case of two summands, $T_0=T_1+T_2$. Since $T_1$ and $T_2$ are rectifiable currents, there exist $\mathcal{H}^{k}$-rectifiable sets $X_1$, $X_2$ and $\mathcal{H}^{k}$-measurable unit simple $k$-vectorfields $\vec T_1$ and $\vec T_2$ and nonnegative integer valued densities $\Theta_i$ such that 
\begin{align*}
\forall i=0,1,2, \quad T_i=(\mathcal{H}^k\res X_i)\wedge\Theta_i \vec T_i,\quad \mc M(T_i)=\int_{X_i}\Theta_i\,d\mathcal{H}^k.
\end{align*}
Furthermore, up to $\mathcal{H}^{k}$-measure $0$, $X_0\subset X_1\cup X_2$. Let 
\begin{align*}
S_1=(\mathcal{H}^k &\res (X_1\setminus X_2))\wedge\Theta_1 \vec T_1, \quad S_2=(\mathcal{H}^k\res (X_2\setminus X_1))\wedge\Theta_2 \vec T_2,\\
\quad S_{12}&=(\mathcal{H}^k\res (X_1\cap X_2))\wedge(\Theta_1 \vec T_1+\Theta_2 \vec T_2).
\end{align*}
Since these currents are supported on disjoint sets,
\begin{align*}
T_0=S_1+S_2+S_{12}, \quad \text{with}\quad \mc M(T_0)=\mc M(S_1)+\mc M(S_2)+\mc M(S_{12}).
\end{align*}
These currents being representable by integration, one can evaluate them on bounded Borel differential forms.
Let $\phi_1,\phi_2,\phi_{12}$ be arbitrary bounded Borel $k-1$-forms supported on $X_1$, $X_2$ and $X_1\cap X_2$ respectively. Up to changing their signs, one can assume that
\begin{align*}
\Scal{S_1 \res\theta}{\phi_1 }\ge 0,\quad \Scal{S_2 \res\theta}{\phi_2 }\ge 0,\quad\Scal{S_{12} \res\theta}{\phi_{12} }\ge 0.
\end{align*}
Then
\begin{align*}
\Scal{S_1 \res\theta}{\phi_1 }&+\Scal{S_2 \res\theta}{\phi_2 }+\Scal{S_{12} \res\theta}{\phi_{12} }\\
&=\Scal{T_0 \res\theta}{\phi_1+\phi_2+\phi_{12} }= 0.
\end{align*}
This implies that
\begin{align*}
\Scal{S_1 \res\theta}{\phi_1 }= \Scal{S_2 \res\theta}{\phi_2 }=\Scal{S_{12} \res\theta}{\phi_{12} }= 0,
\end{align*}
hence
\begin{align*}
S_1 \res\theta=0,\quad S_2 \res\theta=0, \quad S_{12} \res\theta=0.
\end{align*}
On $X_1\cap X_2$, the restriction of $T_0$ is $S_{12}$, so
\begin{align*}
\Theta_0 \vec T_0=\Theta_1 \vec T_1+\Theta_2 \vec T_2
\end{align*}
$\mathcal{H}^k$-almost everywhere on $X_1\cap X_2$. It follows that $\vec T_0=\vec T_1=\vec T_2$ $\mathcal{H}^k$-almost everywhere on $X_1\cap X_2$, hence
\begin{align*}
T_1 \res (X_1\cap X_2)=S_{12} \res \frac{\Theta_1}{\Theta_0}.
\end{align*}
One concludes that
\begin{align*}
T_1 \res \theta&=S_1 \res \theta +(S_{12} \res \frac{\Theta_1}{\Theta_0}) \res \theta\\
&=S_1 \res \theta +(S_{12} \res \theta)\res \frac{\Theta_1}{\Theta_0}=0.
\end{align*}
General case. For every $j_0$, let $T_1=T_{j_0}$ and $T_2=\sum_{j\not=j_0}T_j$. These are rectifiable currents, $\mc M(\TFF)=\mc M(T_1)+\mc M(T_2)$. Applying the special case of two summands, one gets $T_{j_0} \res \theta=0$.
\end{proof}

\subsection{Nearly horizontal $C^1$ chains}

We use Federer's Approximation Theorem to approximate (horizontal) Federer-Fleming integral currents with (nearly horizontal) $C^1$ chains.

\begin{proposition}[Federer's Approximation Theorem]\label{approx}
Let $V$ be a Riemannian manifold. $C^1$ chains are normal mass dense in Federer-Fleming integral currents. Furthermore, supports are under control.
\end{proposition}
In fact, Federer's Approximation Theorem  implies that, for every $\eps>0$, $C^1_{\eps}$ chains are normal mass dense in Federer-Fleming integral currents, but we shall not need this strong form.

One need pass from nearly horizontal to truly horizontal chains. This is performed using a trick due to R. Young, combined with the sharp quantitative estimate on how masses are affected by Heisenberg dilations $\mathfrak s_\lambda$, provided by Lemma \ref{rescale}.

\begin{lemma}
\label{apprhoriz}
For $\lambda\ge 1$, there exist random additive operators $P_\lambda$ and $Q_\lambda$ on $C^1$ chains such that
\begin{enumerate}
  \item $\partial P_\lambda=P_\lambda\partial$.
  \item $P_\lambda+\partial Q_\lambda+Q_\lambda \partial$ is the identity map.
  \item The image of $P_\lambda$ is contained in the subgroup of horizontal $C^1$ chains.
  \item For every $C^1$ chain $S$, 
  \begin{align*}
  \E(\mc M(P_\lambda(S)))&\le C\, (\lambda\,\mc M(S\res\theta)+\mc M(S)+\mc M(\partial S\res\theta)+\lambda^{-1}\mc M(\partial S)),\\
  \E(\mc M(Q_\lambda(S)))&\le C\,(\mc M(S\res\theta)+\lambda^{-1}\mc M(S)).
\end{align*}
\end{enumerate}
\end{lemma}

\begin{proof}
We follow the first step of R. Young's proof of his filling bound, Section 3 in \cite{YoungLowDim}. R. Young uses a periodic triangulation $\tau$ of $\he{n}$. We apply the Random Deformation Theorem and write
$$
S=P_0(S)+\partial Q_0(S)+Q_0(\partial S),
$$
where $P_0(S)$ is a random simplicial chain of $\tau$, $P_0$ and $Q_0$ satisfy linear bounds on expected masses. Then R. Young constructs a periodic self map $\phi$ of $\he{n}$ which is smooth, Lipschitz and horizontal on $n$-simplices of $\tau$. Since $\phi$ moves points a bounded distance away, 
\begin{align*}
1=\phi_\# + \partial Q_1+Q_1\partial,
\end{align*}
where $Q_1$ is additive and deterministically mass-bounded.
Let $P_\phi=\phi_\# \circ P_0$ and $Q_\phi=Q_1+\phi_\# Q_0$.
Then
$$
S=P_\phi(S)+\partial Q_\phi(S)+Q_\phi(\partial S),
$$
where $P_\phi(S)$ is horizontal,
\begin{align*}
 \E(\mc M(P(S)))\le C\, (\mc M(S)+\mc M(\partial S))
\end{align*} and
\begin{align*}
\E(\mc M(Q_\phi(S)))\le C_2 \mc M(S),\quad
\E(\mc M(Q_\phi(\partial S)))\le C_2 \mc M(\partial S).
\end{align*}
This construction is applied to a Heisenberg dilate $(\mathfrak s_\lambda)_{\#}S$, 
\begin{align*}
(\mathfrak s_\lambda)_{\#}S=P_\phi((\mathfrak s_\lambda)_{\#}S)+\partial Q_\phi((\mathfrak s_\lambda)_{\#}S)+Q_\phi(\partial (\mathfrak s_\lambda)_{\#}S),
\end{align*}
leading to
\begin{align*}
S=P_\lambda(S)+\partial Q_\lambda(S)+Q_\lambda(\partial S),
\end{align*}
where $P_\lambda(S):=(\mathfrak s_{1/\lambda})_{\#}P_\phi((\mathfrak s_\lambda)_{\#}S)$ is horizontal and
\begin{align*}
\E(\mc M(P_\lambda S)))&=\lambda^{-k}\E(\mc M(P_\phi((\mathfrak s_\lambda)_{\#}S))\\
&\le C_2 \lambda^{-k}(\mc M((\mathfrak s_\lambda)_{\#}S)+\mc M(\partial (\mathfrak s_\lambda)_{\#}S))\\
&\le C_2 \lambda^{-k}(\lambda^{k+1}\mc M(S\res\theta)+\lambda^{k}\mc M(S)\\&\hskip2cm +\lambda^{k}\mc M(\partial S\res\theta)+\lambda^{k-1}\mc M(\partial S))\\
&\le C_2 (\lambda\,\mc M(S\res\theta)+\mc M(S)+\mc M(\partial S\res\theta)+\lambda^{-1}\mc M(\partial S)).
\end{align*}
$Q_\lambda(S):=(\mathfrak s_{1/\lambda})_{\#}Q_\phi((\mathfrak s_\lambda)_{\#}S)$ satisfies
\begin{align*}
\E(\mc M(Q_\lambda (S)))
&\le \lambda^{-k-1}\E(\mc M(Q_\phi((\mathfrak s_\lambda)_{\#}S)))\\
&\le C_2 \lambda^{-k-1}\mc M((\mathfrak s_\lambda)_{\#}S)\\
&\le C_2 \lambda^{-k-1}(\lambda^{k+1}\mc M(S\res\theta)+\lambda^{k}\mc M(S))\\
&=C_2 (\mc M(S\res\theta)+\lambda^{-1}\mc M(S)),
\end{align*}
as claimed.
\end{proof}

\begin{corollary}\label{intFF=>intmetric}
Let $\TFF $ be a Federer-Fleming horizontal integral current. The current $\TAK=\widetilde{\TFF }$ is a metric integral current and it satisfies $\widehat{\TAK}=\TFF $.
\end{corollary}

\begin{proof}
Let $\TFF $ be a Federer-Fleming horizontal integral current. According to the Approximation Theorem (Proposition \ref{approx}), there exist $C^1$ chains $S_j$ such that $\eps_j:=N(\TFF -S_j)\to 0$. Since $\TFF $ is horizontal,
\begin{align*}
\mc M(S_j\res\theta)\le\eps_j, \quad &\mc M(\partial S_j\res\theta)\le\eps_j,\\
\mc M(S_j)\le \mc M(\TFF )+\eps_j,\quad &\mc M(\partial S_j)\le \mc M(\partial \TFF )+\eps_j.
\end{align*}
Let $\lambda_j=\eps_j^{-1}$. Lemma \ref{apprhoriz} expresses 
\begin{align*}
S_j=P_{\lambda_j}(S_j)+\partial Q_{\lambda_j}(S_j)+Q_{\lambda_j}(\partial S_j),
\end{align*}
with $P_{\lambda_j}(S_j)$ horizontal and
\begin{align*}
\mc M(P_{\lambda_j}(S_j))&\le C\,(\lambda_j\,\mc M(S_j\res\theta)+\mc M(S_j)+\mc M(\partial S_j\res\theta)+\lambda_j^{-1}\mc M(\partial S_j))\\
&\le C\,(1+\mc M(\TFF )+2\eps_j+\eps_j(\mc M(\partial \TFF )+\eps_j)),
\end{align*}
stays bounded, as well as
\begin{align*}
\mc M(\partial P_{\lambda_j}(S_j))&=\mc M(P_{\lambda_j}(\partial S_j))\\
&\le C\,(\lambda_j\,\mc M(\partial S_j\res\theta)+\mc M(\partial S_j))\\
&\le C\,(1+\mc M(\partial \TFF )+\eps_j).
\end{align*}
On the other hand, 
\begin{align*}
\E(\mc M(Q_{\lambda_j}(S_j)))&\le C\,(\mc M(S_j\res\theta)+\eps_j \mc M(\partial S_j))\\
&\le C\eps_j \,(1+\mc M(\TFF )+\eps_j)
\end{align*}
and
\begin{align*}
\E(\mc M(Q_{\lambda_j}(\partial S_j)))&\le C\,(\mc M(\partial S_j\res\theta)+\eps_j \mc M(S_j))\\
&\le C\eps_j (1+\mc M(\partial \TFF )+\eps_j)
\end{align*}
both tend to $0$.

Let us fix random choices which achieve the above expected bounds. This provides us with horizontal $C^1$ chains $P_{\lambda_j}(S_j)$ which converge in FF-flat norm to $\TFF $. According to Lemma \ref{C1=>metric}, $P_{\lambda_j}(S_j)$ is a metric integral current such that $N(P_{\lambda_j}(S_j))$ is bounded. Furthermore, their supports remain in a neighborhood of the support of $\TFF $. By the Compactness Theorem for metric integral currents (Theorems 5.2 and 8.5 in \cite{AK_acta}), some subsequence of $(P_{\lambda_j}(S_j))$ converges weakly to some metric integral current $\TAK$. In particular, it converges on all smooth data $f\,d\pi_1\wedge\cdots\wedge d\pi_k$, so $\widehat{\TAK}=\TFF $ and $\TAK=\widetilde{\TFF }$. This shows that the correspondence between normal currents in both theories maps integral currents to integral currents. 

\end{proof}

\section*{Acknowledgements}

B.F. is supported by the University of Bologna, funds for selected research topics.

P.P. is supported by Agence Nationale de la Recherche, ANR-22-CE40-0004 GOFR.

\bibliographystyle{amsplain}

\bibliography{currents_25_11_20_PP}

\bigskip
\tiny{
\noindent
Bruno Franchi 
\par\noindent
Universit\`a di Bologna, Dipartimento
di Matematica\par\noindent Piazza di
Porta S.~Donato 5, 40126 Bologna, Italy.
\par\noindent
e-mail:
bruno.franchi@unibo.it.
}

\medskip

\tiny{
\noindent
Pierre Pansu 
\par\noindent  Universit\'e Paris-Saclay, CNRS, Laboratoire de math\'ematiques d'Orsay
\par\noindent  91405, Orsay, France.
\par\noindent 
e-mail: pierre.pansu@universite-paris-saclay.fr
}

\bigskip

Keywords: current; Heisenberg group; sub-Riemannian metric; contact manifold; differential form

MSC 2000: 
28A75; 
53C17; 
53C42; 
53D10; 
58A10; 
58A25 

\end{document}